\documentclass[11pt,a4paper]{article}
\usepackage{geometry}
\usepackage{pdfpages}
\usepackage{graphicx}
\usepackage{bm,array}
\usepackage{comment}
\usepackage{caption}
\usepackage{subcaption}
\usepackage{tablefootnote}
 \usepackage{makecell}
\usepackage{booktabs}
 \usepackage{multirow}
 \usepackage{float}
 \usepackage{blindtext}
 \usepackage{mathtools}

\usepackage[normalem]{ulem}

\usepackage[pdfpagemode={UseOutlines},bookmarks=true,bookmarksopen=true, bookmarksopenlevel=0,bookmarksnumbered=true,hypertexnames=false, colorlinks,linkcolor={blue},citecolor={blue},urlcolor={red},pdfstartview={FitV},unicode,breaklinks=true]{hyperref}
\hypersetup{urlcolor=blue, colorlinks=true}

\usepackage{setspace}
\usepackage{vmargin}
\setmarginsrb{ 1.0in}  
{ 0.6in}  
{ 1.0in}  
{ 0.8in}  
{  20pt}  
{0.25in}  
{9pt}  
{ 0.3in}  

\usepackage{authblk}

\usepackage{amsmath}
\usepackage{amsfonts}
\usepackage{amssymb}
\usepackage[round, sort,comma,authoryear]{natbib}

\newtheorem{theorem}{Theorem}

\newtheorem{lemma}{Lemma}

\newtheorem{proposition}{Proposition}

\newenvironment{proof}[1][Proof]{\noindent\textbf{#1.} }{\ \rule{0.5em}{0.5em}}

\usepackage{multirow}
\usepackage{hhline}
\usepackage{footnote}
\usepackage[ruled,vlined,linesnumbered]{algorithm2e}
\usepackage{algorithmic}
\newcommand{\comments}[1]{{\color{blue}\textit{$\#$ #1}}}

\newcommand{\MNL}{\textsc{MNL}}
\newcommand{\transpose}{{\mbox{\tiny T}}}

\newcommand{\obx}{\overline{\textbf{x}}}
\newcommand{\oby}{\overline{\textbf{y}}}
\newcommand{\obz}{\overline{\textbf{z}}}

\newcommand{\cG}{{\mathcal{G}}}
\newcommand{\cV}{{\mathcal{V}}}

\newcommand{\cX}{{\mathcal{X}}}

\newcommand{\cN}{{\mathcal{N}}}

\newcommand{\cT}{{\mathcal{T}}}

\newcommand{\cC}{{\mathcal{C}}}

\newcommand{\cO}{{\mathcal{O}}}

\newcommand{\bW}{\textbf{W}}

\newcommand{\bx}{\textbf{x}}
\newcommand{\by}{\textbf{y}}

\newcommand{\ba}{\textbf{a}}

\newcommand{\bz}{\textbf{z}}

\newcommand{\btheta}{\pmb{\theta}}

\newcommand{\bbR}{\mathbb{R}}

\usepackage{xspace}
\newcommand{\CSE}{[\textsc{C-Separated}]\xspace}
\newcommand{\CSH}{[\textsc{C-Sharing}]\xspace}
\newcommand{\CNL}{\textsc{\tiny CNL}}
\newcommand{\NL}{\textsc{\tiny NL}}

\newif\ifnotes\notestrue
\def\tt#1{\textbf{\texttt{#1}}}

\def\htien#1{}

\usepackage{xcolor}

\begin{document}







\newcolumntype{C}{>{\centering\arraybackslash}p{4em}}

\title{\textbf{Competitive Facility Location under Cross-Nested Logit Customer Choice Model: Hardness and Exact Approaches}}
\author[1,2]{Ba Luat Le}
\author[2]{Tien Mai}
\author[1,*]{Thuy Anh Ta}
\author[3,4]{Minh Hoang Ha}
\author[3,4]{Duc Minh Vu}
\affil[1]{\it\small
ORLab, Faculty of Computer Science, Phenikaa University, Ha Dong, Hanoi, VietNam}
\affil[2]{\it\small
School of Computing and Information Systems, Singapore Management University, 80 Stamford Rd, Singapore 178902}
\affil[3]{\it\small
Faculty of Data Science and Artificial Intelligence, NEU College of Technology, Hai Ba Trung, Hanoi, VietNam}
\affil[4]{\it\small
ORLab-SLSCM, National Economics University,  Hai Ba Trung, Hanoi, VietNam}
\affil[*]{\it\small
Corresponding author, anh.tathuy@phenikaa-uni.edu.vn
}

\maketitle


\begin{abstract}
We study the competitive facility location problem, where a firm aims to establish opening facilities in a market already occupied by competitors. In this problem, customer behavior is a crucial factor in making optimal location decisions. We explore a general class of customer choice models, known as the cross-nested logit model, which is recognized for its flexibility and generality in predicting people's choice behavior. 
To explore the problem, we first demonstrate that it is \texttt{NP-Hard}, even when there is only one customer class. We further show that this hardness result is tight, as the facility location problem under any simpler choice models (such as the logit or nested logit) is polynomial-time solvable when there is one customer class.

To tackle the resulting facility location problem, we demonstrate that the objective function under a general cross-nested structure is not concave. Interestingly, we show that by a change of variables, the objective function can be converted to a convex program (i.e., a maximization problem with a concave objective and convex constraints), enabling it to be solved to optimality via an outer-approximation algorithm. Extensive experiments show the efficiency of our approach and provide analyses on the benefits of using the cross-nested model in the facility location context.

\end{abstract}


\section{Introduction}

Facility location has long been an important problem in decision-making for modern transportation and logistics systems. Traditionally, this problem involves selecting a subset of potential locations from a given pool of candidates and deciding on the financial resources allocated to establish opening facilities at these chosen locations. The objective is to either maximize profit (such as expected customer demand or revenue) or minimize costs (such as operational or transportation expenses). Customer demand plays a crucial role in these decisions, making it a key factor in facility location problems. In this work, we focus on a specific category of competitive facility location problems, where customer demand is characterized and forecasted using a random utility maximization (RUM) model \citep{train2009discrete,BenaHans02,mai2020multicut}. In this setting, it is assumed that customers make their choices among available facilities by maximizing their utilities associated with each facility. These utilities typically depend on attributes (or features) of the facilities, such as service quality, infrastructure, or transportation costs, or characteristics of customers, such as age, income, and gender. The application of the RUM model framework in this context is well justified by the popularity and proven success of RUM models in modeling and forecasting human choice behavior in transportation-related applications \citep{Mcfadden2001economicNobel,BenABier99a}.

In the context of competitive facility location under RUM models, prior work primarily utilizes the classical multinomial logit (MNL) model, which is one of the most popular choice models in demand modeling. Although this model is convenient to use due to its simple structure, it is limited by its Independence of Irrelevant Alternatives (IIA) property, implying that the ratio between the choice probabilities of any two facilities is independent of any other facilities. However, this property often does not hold in practice, limiting the application of the RUM framework in accurately modeling customer behavior.
Let us give an example to illustrate why the IIA property may not hold in the context of facility location. Suppose there are three facilities: A, B, and C, where A is located in a downtown area, and B and C are the same area and are outside of the downtown area. If Facility C is closed, the demand would shift to the remaining two facilities. It is intuitive to see that this closing may lower the attractiveness of the not-downtown area (as now there is only one facility in that area), thus affecting the attractiveness of Facility A differently than that of Facility B (as B and C are in the same area), implying that the IIA property does not hold. 

Several papers address this limitation by considering more advanced choice models, such as the nested logit (NL) model, the Generalized Extreme Value (GEV) model, or the mixed-logit model. The NL model, for instance, groups locations into different disjoint subsets, partially relaxing the IIA property. However, the IIA property still holds for alternatives from different nests. The work in \citep{dam2022submodularity,dam2023robust} considers a facility location problem under a generalized class of models that encompasses almost every RUM model in the literature. However, this work only offers heuristic methods that cannot guarantee optimal solutions. Moreover, while the mixed-logit model is known for its generality and has been considered in some prior studies \citep{FLO_Hasse2009MIP,haase2014comparison}, it often requires a large number of samples to accurately approximate the choice probabilities. This results in large facility location problem instances that are expensive to solve.

In this paper, we study a competitive facility location problem under the cross-nested logit (CNL) model, which is known to be one of the most flexible and general models within the RUM family. This model generalizes the NL model by allowing choice alternatives to belong to multiple subsets (or nests) that are not necessarily disjoint. By fully relaxing the IIA property, the CNL model can approximate any RUM model arbitrarily closely \citep{fosgerau2013choice,bierlaire2006theoretical}. To the best of our knowledge, this is the first time this general model is specifically considered in the context of competitive facility location. Consequently, there are no existing methods capable of solving the facility location problem under the CNL to optimality. We address this research gap in this paper. 

\noindent \textbf{Contributions:}  Specifically, we make the following contributions:

\begin{itemize}
\item[(i)] \textbf{{Problem Formulation and Hardness Results.}} We first formulate the competitive facility location problem and examine its complexity. We theoretically show that the facility location problem under the CNL model is \texttt{NP-hard} even when there is only one customer class (i.e., one demand point) and the cross-nested structure comprises only two nests. We also demonstrate that this hardness result is tight, in the sense that the facility location problem with one customer class under any simpler choice models (such as the NL or MNL models) is solvable in polynomial time. To the best of our knowledge, this is the first time this hardness has been rigorously explored in the context of competitive facility location under RUM models.
\item[(ii)]\textbf{{Convexity and Exact Solution Methods.}} To explore solution methods for solving the proposed problem, we consider two configurations of the cross-nested structure. The first setting, named \CSE, refers to the scenario where competitors' facilities are separated from opening facilities within the cross-nested structure. The second setting, called \CSH, involves a configuration where competitors share a common cross-nested structure with the opening facilities. The rationale behind this consideration is that while \CSE is typically employed in prior work involving facility location under the NL or GEV models \citep{dam2022submodularity,mendez2023follower}, we propose the \CSH setting and argue that it is more general and exhibits a more flexible correlation structure between facility utilities. 

We then demonstrate that while the facility location problem under \CSE has a concave objective function, making standard outer-approximation methods applicable for solving it to optimality, the objective function under \CSH is not concave. Interestingly (and surprisingly), we show that by a change of variables, the problem under \CSH can be reformulated as a mixed-integer convex program (i.e., a maximization problem with a concave objective function and convex constraints). This reformulation allows us to solve the originally non-convex problem exactly using outer-approximation methods. We then show how outer-approximation and submodular cuts can be utilized in cutting plane (CP) or Branch-and-Cut (B\&C) methods to efficiently solve the problem.
\item[(iii)]\textbf{{Numerical Analysis.} }We conduct extensive experiments using popular benchmark instances from the literature to assess the performance of our methods. Our results indicate that our methods (i.e., the convex reformulations solved by CP or B\&C methods) are capable of solving almost all the instances under consideration to optimality and outperform other standard baselines. We also provide experimental analysis on the impact of different cross-nested parameters on the performance of our algorithms, as well as the benefits of using the CNL model compared to the NL or MNL models in the context of competitive facility location. 
  \end{itemize}

\noindent \textbf{Paper Outline:} The paper is organized as follows. Section \ref{sec:review} provides a literature review. Section \ref{sec:formulation and hardness} discusses the problem formulation and Section \ref{sec:hardness} discusses some hardness results. 
Section \ref{sec:methods} presents our solution methods. Section \ref{sec:experiments} presents our experimental results, and finally, Section \ref{sec:conclusion} concludes the paper. Additional experiments are included in the appendix.

\noindent
\textbf{Notation:}
Boldface characters represent matrices (or vectors), and $a_i$ denotes the $i$-th element of vector $\ba$ is it is indexable. We use $[m]$, for any $m\in \mathbb{N}$, to denote the set $\{1,\ldots,m\}$.

\section{Literature Review}\label{sec:review}

In the context of the competitive facility location (CFL) under RUM models, most existing studies employ the MNL model to capture customer demand. For instance, \cite{BenaHans02} appears to be the first to propose the CFL under the MNL model. Their solution method involves a Mixed-Integer Linear Programming (MILP) approach based on a branch-and-bound procedure for small instances and a simple variable neighborhood search for larger instances. Subsequently, alternative MILP models were introduced by \cite{Zhang2012} and \cite{Haase2009}. \cite{haase2014comparison} benchmark these MILP models and concluded that the formulation by \cite{Haase2009} shows the best performance. \cite{Freire2015} strengthened the MILP formulation of \cite{Haase2009} using a branch-and-bound algorithm with tight inequalities. \cite{Ljubic2018outer} proposed a Branch-and-Cut method that integrates outer-approximation and submodular cuts, while \cite{mai2020multicut} developed a multicut outer-approximation algorithm for efficiently solving large instances. In this approach, outer-approximation cuts are generated for each group of demand points rather than for individual demand points.

It is important to note that the MNL model is limited by its IIA property, which implies that the ratio of the probabilities of choosing two facilities is independent of any other alternative. This property does not hold in many practical contexts \citep{train2009discrete,McFaTrai00}.  Efforts have been made to overcome this limitation, including the Mixed Logit Model (MMNL) \citep{McFaTrai00}, models in the Generative Extreme Value (GEV) family such as the Nested Logit \citep{BenA73,BenALerm85}, and the Cross-Nested Logit \citep{VovsBekh98}. 
In the context of CFL, there are also a couple of studies investigating the CFL under the MMNL model \citep{Haase2009,haase2014comparison} with a note that the use of the MMNL often requires large samples to approximate the objective function, leading to large problem instances to solve. \cite{dam2022submodularity,dam2023robust} seem to be the first to incorporate the general GEV family into the CFL, proposing a heuristic method that outperforms existing exact methods. \cite{mendez2023follower} explicitly study a CFL problem under the NL model and propose exact methods based on outer-approximation and submodular cuts incorporated into a B\&C procedure. Besides the fact that NL only partially relaxes the IIA property (i.e., the IIA property still holds across nests), the work of \cite{mendez2023follower} relies on the \CSE configuration (i.e. competitor's facilities are separated with a new one within the cross-nested structure) which is limited in capturing the correlation between competitors' and opening facilities. Our work builds upon the CNL model, which is able to fully relax the IIA property, allowing for a more flexible correlation structure between choice alternatives. Moreover, we explicitly consider the \CSH configurations, which are, as shown later, much more general and adequate than the \CSE. Our solution methods also guarantee optimal solutions, making a significant advancement over the prior heuristics and exact methods  \citep{dam2023robust,dam2022submodularity, mendez2023follower}.

It is worth mentioning that our work is related to a body of research on competitive facility location where customer behavior is modeled using gravity models \citep{DREZNER2002, ABOOLIAN2007a, ABOOLIAN2007b, ABOOLIAN2021}. These models, while based on a different type of customer choice models, share a similar objective structure with the CFL problem under the MNL (i.e., the objective function is a sum of linear fractions). There is also a line of work considering CFL with a general RUM model, with solution methods based on sample average approximation (SAA) \citep{lamontagne2023optimising, legault2024model}. While these approaches are general, the approximation would require a large number of samples to achieve acceptable approximation errors, leading to large problem instances to solve.

Regarding the CNL customer choice model, the model was mostly introduced explored in the transportation research community \citep{small1987discrete, vovsha1997application, ben1999discrete}. \citet{bierlaire2006theoretical} provided a theoretical foundation for the CNL model, formally proving its inclusion in the GEV family \citep{McFa78} and introducing a novel estimation procedure based on non-linear programming, as opposed to the heuristics used in previous studies. This work also established that the MNL and NL models are special cases of the CNL model. 
It is well-known that the CNL model is fully flexible, in the sense that it can approximate arbitrarily closely any RUM model  \citep{fosgerau2013choice}, as well as the general ranking preference model \citep{aouad2018approximability, LeC2024POM}.
Thanks to its flexible structure, the CNL model has been successfully applied to a wide range of transportation problems. For example, applications of the CNL model can be found in mode choice problems \citep{yang2013cross}, departure time choice \citep{ding2015cross}, route choice and revenue  management\citep{lai2015specification,mai2016method,mai2017dynamic}, air travel management \citep{drabas2013modelling}, and more recently, location choice in international migration \citep{beine2021new}. 

It has also been observed that the CNL model consistently outperforms other choice models, such as MNL and NL, in demand modeling. For instance, \cite{beine2021new} used migration aspiration data from India and demonstrated that the CNL outperforms competing approaches (i.e., the MNL, NL, and MMNL) in terms of quality of fit and predictive power, highlighting stronger heterogeneity in responses to shocks and revealing complex and intuitive substitution patterns. Similarly, \cite{ding2015cross} used revealed preference data collected from Maryland-Washington, DC, and showed that a CNL structure offers significant improvements over MNL and NL models.

Despite its success in transportation research, the potential of the CNL model in decision-making remains largely unexplored. To the best of our knowledge, this study is the first to employ the CNL model to capture substitution behavior in facility location. Our work is closely related to the work of \cite{LeC2024POM}, which examines an assortment optimization problem under the CNL model with an objective function that shares some similarities with the one considered in this paper. In \cite{LeC2024POM}, the authors propose a near-optimal approach based on piecewise linear approximation that only guarantees a near-optimal solution, with the size of the approximate problem being proportional to the solution precision. Consequently, one would theoretically need to solve instances of infinite size to guarantee an optimal solution.



\section{Competitive Facility Location under CNL Model}\label{sec:formulation and hardness}

In this section, we present the problem formulation of the competitive facility location problem under the CNL model. We then discuss the CNL model's ability to capture complex correlation structures between choice alternatives, highlighting its advantages over other choice models considered in the literature (i.e., the MNL and NL models).

\subsection{Problem Formulation}
In the classic facility location problem, decision-makers aim to establish opening facilities in a way that optimizes the demand fulfilled by customers. However, accurately assessing customer demand in real-world scenarios is challenging and inherently uncertain. In this work, we explore a facility location problem where discrete choice models \citep{train2009discrete} are used to estimate and predict customer demand. Among the various approaches discussed in the demand modeling literature, the RUM framework \citep{train2009discrete} stands out as the most prevalent method for modeling discrete choice behaviors. This method is grounded in random utility theory, positing that a decision-maker's preference for an option is represented through a random utility. Consequently, the customer tends to choose the alternative offering the highest utility. According to the RUM framework \citep{McFa78,FosgBier09}, the likelihood of individual \(t\)  choosing an option {\(i\) from a given choice set $S$} is determined by \(P(u_{ti} \geq u_{tj}, \; \forall j \in S)\), implying that the individual will select the option providing the highest utility. Here, the random utilities are typically defined as \(u_{ti} = v_{ti} + \epsilon_{ti}\), where \(v_{ti}\) represents the deterministic component, which can be calculated based on the characteristics of the alternative and/or the decision-maker and some parameters to be estimated, and \(\epsilon_{ti}\) represents random components that are unknown to the analyst. Under the popular MNL model, the probability that a facility located at position \(i\) is chosen by an individual \(t\) is computed as 
$$P_t(i|S) = \frac{e^{v_{ti}}}{\sum_{i \in S} e^{v_{ti}}},$$ where \(S\) is a set of available facilities.

In this study, we consider a competitive facility location problem where a ``newcomer" company plans to enter a market already captured by a competitor. For example, in the context of retail chains, it is a scenario where a new supermarket chain plans to enter a market dominated by established players. The new entrant aims to strategically locate its stores to attract customers away from existing supermarkets. The main objective is to secure a portion of the market share by attracting customers to their newly opened facilities. 
The company's strategy revolves around selecting an optimal set of locations for its opening facilities to maximize the anticipated customer demand.

In this work, we consider the CNL model, which is known to be one of the most flexible and general RUM models in the literature. To formulate the problem, let \( [m] = \{1, 2, \ldots, m\} \) be the set of all available locations, and let \( [T] = \{1,\ldots,T\}\) be the set of customer types, which can be defined based on customers' geographical locations or characteristics such as age, income, and gender. Also, let \( \mathcal{C} \) be the set of competitors' existing facilities. Under the CNL model, the set of locations can be assigned to different subsets (or nests). For each customer type $t\in [T]$, we assume that \( [m] \cup \mathcal{C} \) can be assigned to \( N \) subsets (or nests) \( \cN^t_1, \cN^t_2, \ldots, \cN^t_N\)  according to their attributes/characteristics. Note that \( \cN^t_1, \cN^t_2, \ldots, \cN^t_N \) are not necessarily disjoint, i.e., a facility/location \( i \) can belong to multiple nests. We also use a non-negative quantity \( \alpha^t_{in} \) to capture the level of membership of location \( i \) in nest \( \cN^t_n \), where \( \sum_{n \in [N]} \alpha^t_{in} = 1 \) for all \( i \in [m] \). Without loss of generality, we can assume that \( \cN^t_n = [m]\cup \cC \) for all \( n \in [N] \), and if location \( i \) does not belong to \(\cN^t_n \), we can set \( \alpha^t_{in} = 0 \) without affecting the previous assumption. Let $v_{ti}$ be the deterministic utility associated with location $i\in [m]\cup\cC$  and customer type $t\in [T]$. Such a utility function can be specified as a function of customers' and locations' characteristics, with parameters that can be inferred by estimating the choice model \citep{train2009discrete}. In this work, we assume that \( v_{ti} \) are given for all $t\in [T], i\in [m]$. 

In Figure \ref{fig:sample-figure}, we provide an illustration of a cross-nested correlation structure for five facilities (both new and competitors' facilities). Due to their characteristics, these facilities can be grouped into different subsets (or nests). For instance, Facility 1 and Facility 3 can be grouped into a nest representing the ``Downtown Area'', as they are both located in the downtown area. Additionally, Facility 2 and Facility 3 can be assigned to Nest 3, representing the availability of public transport. It can be seen that such nests should not be disjoint, as each facility can belong to different categories.  
\begin{figure}[htb] 
    \centering
    \includegraphics[width=\textwidth]{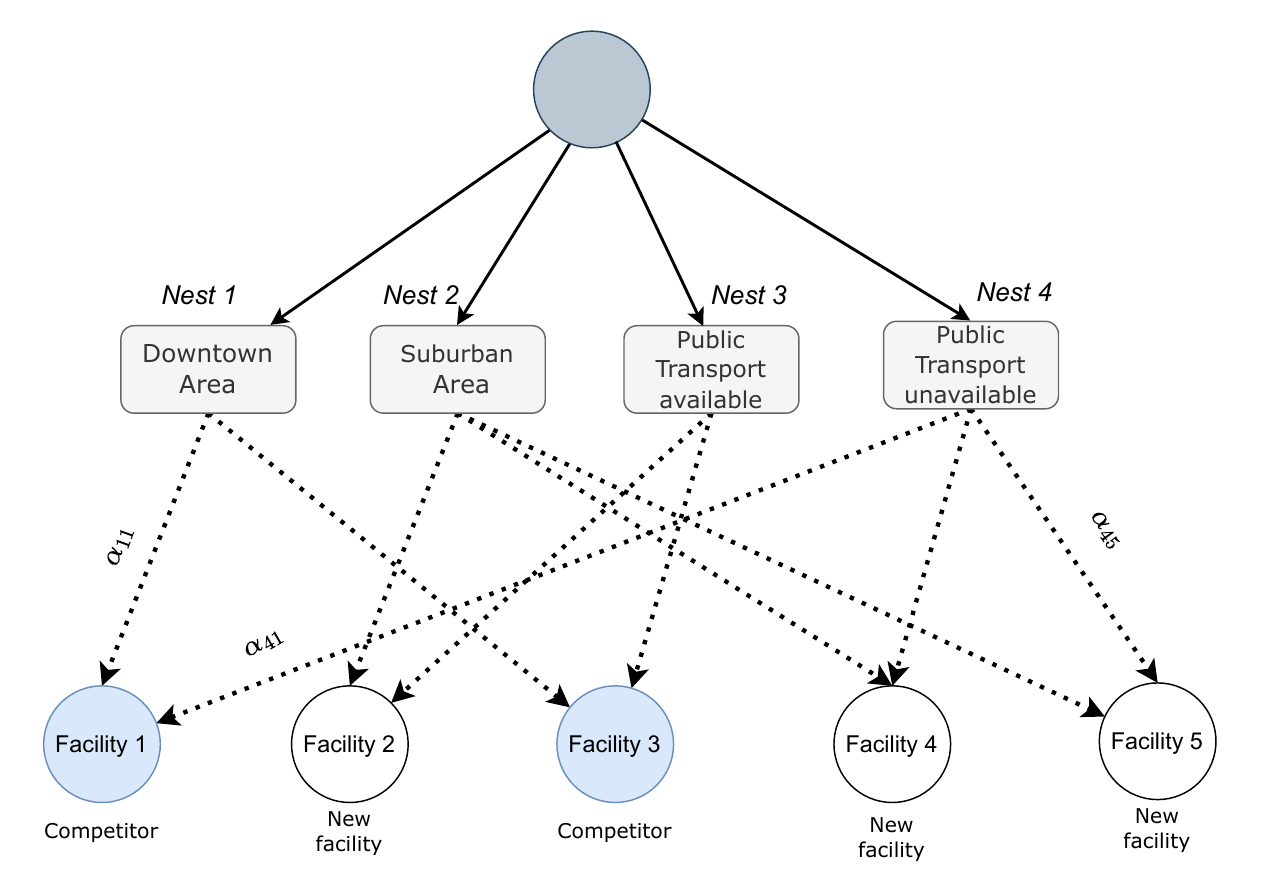} 
    \caption{An example of a cross-nested structure in facility location.}
    \label{fig:sample-figure} 
\end{figure}

The choice process in the CNL model can viewed as a two-stage process where a customer would first select a nest, and then select a location/facility within this nest. {Consider a set of available locations $S\subseteq [m]$,} the probability that a customer would select a nest $\cN^t_n$ , for any $n\in [N]$ can be computed as
\[
P(\cN^t_n|S) = \frac{W_{tn}^{\sigma_{tn}}}{ \sum_{n'\in[N]}W_{tn'}^{\sigma_{tn'}}}, 
\]
where $W_{tn} = \sum_{i\in S \cap ([m]\cup \cC)}\alpha^t_{in}e^{v_{ti}/\sigma_{tn}}$ is the total preference utilities of all alternatives in {$\cN^t_n \cap S$}, $\sigma_{tn}$ is the dissimilarity parameter of nest $\cN^t_n$. It is typically assumed that the value of $\sigma_{tn}$ varies in the unit interval for all nests to guarantee that the model is consistent with the RUM framework \citep{McFa78,bierlaire2006theoretical}. 

In the second stage, the customer decides to select a location/facility $i\in S$ from the chosen nest $\cN^t_n$ with probabilities:
\[
P(i|\cN^t_n) = \frac{\alpha^t_{in}e^{v_{ti}/\sigma_{tn}}}{W_{tn}}, \forall i\in S
\]
For ease of notation, let $V_{tin} = e^{v_{ti}/\sigma_{tn}}$. Then, the probability that a customer of type $t\in [T]$ select a facility in $S$ can be computed as:
\begin{align*}
		P^t(i|S) = \sum_{n\in [N]}P(\cN^t_n|S) P(i|\cN^t_n) &= \sum_{n\in[N]}\frac{W_{tn}^{\sigma_{tn}}}{ \sum_{n'\in[N]}W_{tn'}^{\sigma_{tn'}}}\times \frac{\alpha^t_{in}V_{tin}}{W_{tn}}\\
		&=\frac{\sum_{n\in [N]}{W_{tn}}^{\sigma_{tn}-1} (\alpha^t_{in}V_{tin})}{ \sum_{n\in [N]} W_{n}^{ \sigma_{tn}}},~\forall i\in S.
	\end{align*}
Now, by the law of total expectation, the expected captured market share given  by the selected locations $S$ can be computed as 
\begin{align}
    F(S) = \sum_{t\in [T]} q_t\frac{\sum_{n\in [N]}{W_{tn}}^{\sigma_{tn}-1} (\sum_{i\in S}\alpha^t_{in}V_{tin})}{ \sum_{n\in [N]} W_{tn}^{ \sigma_{tn}}},\nonumber
\end{align}
where $q_t$ is the total demand of customers of type $t$. With the objective of maximizing total expected demand, the competitive facility location  problem under the CNL model can be formulated as:
\[
\max_{S:|S|\leq r} F(S).
\]
Given this objective, the problem is also referred to as the Maximum Capture Problem (MCP) \cite{BenaHans02}.
It is also convenient to formulate the problem as a binary nonlinear program. To this end, let $\bx = (x_1, x_2, \ldots, x_m) \in \{0, 1\}^m$ be a binary vector representing a location choice  decision where $x_i = 1$ if, and only if, the location $i$ is chosen. We can formulate the MCP as the following  binary nonlinear program:
	\begin{align}
 		\max_{\bx}\qquad &\left\{F(\bx)  = \sum_{t\in [T]}q_t\frac{\sum_{n\in [N]}{W_{tn}}^{\sigma_{tn}-1} (\sum_{i\in [m]}\alpha^t_{in}V_{tin} x_i)}{ \sum_{n\in [N]} W_{tn}^{ \sigma_{tn}}}\right\}\label{prob:CNL-MCP}\tag{\sf MCP-CNL} \\
        \text{subject to} \quad  & \sum_{i\in [m]} x_i \leq r\nonumber\\ 
        & \bx \in \{0,1\}^m\nonumber
	\end{align}
 where $W_{tn} =\sum_{i\in [m]}\alpha^t_{in}x_iV_{tin} + \sum_{i\in \cC}\alpha^t_{in}V_{tin}, \forall{t\in [T], n \in [N]}.$
It can be observed that the objective function \( F(\mathbf{x}) \) is highly nonlinear and not necessarily concave in \(\mathbf{x}\) (or no one has proved that it is concave), making exact approaches for mixed-integer nonlinear programs, such as MILP, conic reformulations \citep{Sen2017}, or outer-approximation \citep{duran1986outer,mai2020multicut}, inapplicable. To the best of our knowledge, there are no exact methods capable of solving \eqref{prob:CNL-MCP} to optimality. We will present our solution in the next section. However, before this, let us delve into the correlation structure of the CNL model and discuss some NP-hardness results.

\subsection{Cross-nested Correlation Structure}

Previous studies examining the MCP under the NL or GEV models are limited by the assumption that competitors' facilities do not share the same nests as the opening facilities \citep{dam2022submodularity,dam2023robust,mendez2023follower}. Specifically, prior formulations assume that nests are constructed to contain only opening facilities. This is a significant limitation in a modeling point of view because, under this setting, competitors' facilities are uncorrelated (in terms of utility) with the opening facilities. Furthermore, as will be shown later, this setting also leads to simpler problems to solve. We delve into this issue in this subsection. 

First, to facilitate our exposition, let us define the following two settings:

\begin{itemize}
    \item \textbf{Competitors' facilities are separated in the nested structure } (or \CSE for short): This refers to the setting where no facilities from the competitor share the same nest with any new facility.
    \item \textbf{Competitors' facilities are shared in the nested structure } (or \CSH for short): This refers to the setting where the competitor and opening facilities can belong to the same nests.
\end{itemize}

In Figure \ref{fig:CSE-CSH}, we provide an illustration of the two configurations, \CSE and \CSH. The \CSH configuration is more general, allowing new and competitors' facilities to belong to any common nests. In contrast, the \CSE configuration only permits new and competitors' facilities to belong to separate nests. 
This distinction makes \CSE highly limited from a modeling perspective, as it restricts the ability to accurately represent the relationships between available facilities in the market. By enforcing strict separation of facilities into different nests, \CSE fails to capture the nuanced interactions and overlaps that exist in real-world scenarios. On the other hand, \CSH's flexibility in assigning facilities to multiple nests enables a more comprehensive and realistic representation of market relationships.

  \begin{figure}[htb] 
    \centering
    \includegraphics[page=1,width=\textwidth]{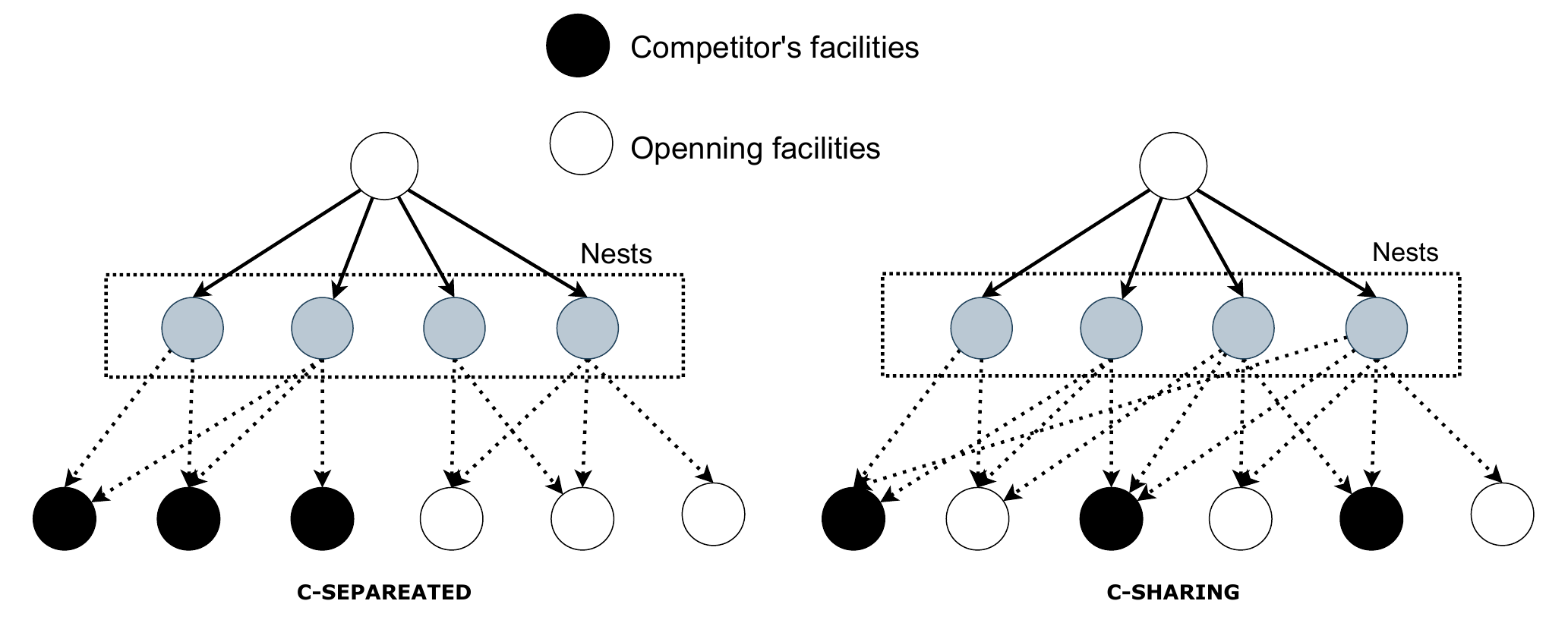} 
    \caption{Illustrations of \CSE and \CSH.
    \label{fig:CSE-CSH}} 
\end{figure}
To understand why \CSE is significantly limited compared to \CSH, let us examine the correlation structure provided by the two settings. Under the CNL model, it is known that the correlation between two utilities can be approximated as \citep{papola2004some, abbe2007normalization}:
\[
\widehat{\textsf{Corr}}^{\CNL}(U_{ti}, U_{tj}) = \sum_{n \in [N]} \sqrt{\alpha^t_{in} \alpha^t_{jn}} (1 - \sigma^2_{tn}),
\]
where \(U_{ti}\) is the random utility of facility/location \(i \in [m] \cup \mathcal{C}\). This formulation implies that any pair of locations/facilities can be correlated in utility, as long as there is a nest covering them. This leads to the observation that, under the \CSE setting, \(\widehat{\textsf{Corr}}^{\CNL}(U_{ti}, U_{tj}) = 0\) for any new facility \(i \in [m]\) and competitor's facility \(j \in \mathcal{C}\). This would pose a modeling issue as a competitor's facility \(j\) can be highly correlated in utility with a new facility due to, for instance, geographical proximity. For example, if a competitor's facility is close to a new facility, their utilities may be highly and positively correlated. To illustrate this, consider a scenario where \(i\) and \(j\) are nearby. Due to a reason such as the construction of a new road or parking area, the utilities of both \(i\) and \(j\) would significantly increase,  implying that the correlation (in utility) between facility/location \(i\) and \(j\) would be even stronger than between \(i\) and any new facility in different areas.

It is also worth mentioning the correlations in other choice models such as the MMNL and NL models. Due to the IIA property, the MNL model has a diagonal variance-covariance matrix, meaning product utilities are uncorrelated, i.e., 
\[
\textsf{Corr}^{\textsc{MNL}}(U_{ti}, U_{tj}) = 0,~\forall i,j\in [m], i\neq j.
\]
In the NL model, only products belonging to the same nest are correlated, and the correlation between product utilities is given by \citep{train2009discrete, abbe2007normalization}:
\[
\textsf{Corr}^{\textsc{NL}}(U_{ti}, U_{tj}) = (1 - \sigma^2_{tn}) \delta_{tn}(i, j),
\]
where \(\delta_{tn}(i, j) = 1\) if \(i\) and \(j\) are both in nest \(\mathcal{N}^t_n\), and is equal to zero otherwise. This implies that while facilities are positively correlated (in utility) with those in the same nest, they are uncorrelated with any products outside of their nest. The CNL model, which allows each product to belong to multiple nests, yields more flexible correlation structures.

We now discuss the problem formulations under the two cross-nested structural settings: \CSE and \CSH. While the objective function under the \CSH setting is given in \eqref{prob:CNL-MCP}, under the \CSE setting, the objective function of the MCP can be rewritten as:
\begin{align}
 F(\mathbf{x}) = \sum_{t \in [T]} q_t \frac{\sum_{n \in [N]} W_{tn}^{\sigma_{tn} - 1} \left( \sum_{i \in [m]} \alpha^t_{in} V_{tin} \right)}{U^c_t + \sum_{n \in [N]} W_{tn}^{\sigma_{tn}}} \label{eq:obj-CSH}
\end{align}

where \(W_{tn} = \sum_{i \in [m]} \alpha^t_{in} V_{tin}\) and \(U^c_t\) represents the sum of exponentials of the utilities of all the competitors' facilities. Here we note that the denominator can be separated into two terms: the first term \(U^c_t\) only depends on the competitors' utilities, and the second term \(\sum_{n \in [N]} W_{tn}^{\sigma_{tn}}\) only depends on the utilities of the opening facilities. This is possible because the competitors' facilities are separated in the nesting structure.

We now can further write the objective function under \CSE as:
\begin{equation}\label{eq:obj-CSE}
    F(\mathbf{x}) = \sum_{t \in [T]} q_t \frac{\sum_{n \in [N]} W_{tn}}{U^c_t + \sum_{n \in [N]} W_{tn}^{\sigma_{tn}}}
= \sum_{t \in [T]} q_t - \sum_{t \in [T]} \frac{U^c_t}{U^c_t + \sum_{n \in [N]} W_{tn}^{\sigma_{tn}}}
\end{equation}
This involves a sum of fractions, where each fraction has a constant numerator. This type of objective function is often seen in prior works \citep{BenaHans02,Ljubic2018outer,mai2020multicut,dam2022submodularity} in the context of the MCP, where the authors have proved that their objective function is concave in \(\bx\). As shown in the next section, it is also the case under the CNL model. Compared to the objective function under \CSH shown in \eqref{prob:CNL-MCP}, due to the fact that the competitors share the same nesting structure with the opening facilities, it seems not possible to convert the objective function into a sum of fractions with constant numerators. As a result, the objective in \eqref{prob:CNL-MCP} is much more challenging to handle compared to the one in \eqref{eq:obj-CSH}. In fact, as shown in the subsequent sections, the objective function in \eqref{prob:CNL-MCP} is generally not concave, making some standard approaches, such as outer-approximation methods, not directly applicable.

\section{NP-harness}\label{sec:hardness}

In this section, we discuss the difficulty of solving the MCP under the CNL model. Our main result is that the MCP-CNL is \texttt{NP-hard} even when there is only one customer type (i.e., \(T=1\)). To demonstrate the tightness of this hardness result and to provide a broader context for our findings, we first consider the MCP under the MNL and NL models, which are two specific cases of the CNL model commonly used in the literature. We show that when there is only one customer type (i.e., \(T=1\)), the MCP under the MNL or NL models is solvable in polynomial time.

The following proposition demonstrates that the MCP is polynomial-time solvable under the MNL model when there is only one customer type but becomes NP-hard when \(T \geq 2\). To prove these results, for \(T=1\), we show that the MCP is equivalent to solving a linear program with a cardinality constraint, which can be efficiently solved using a simple sorting algorithm with a runtime of \(\mathcal{O}(m \log m)\). For the case of \(T=2\), we connect the problem to the \textit{Set Partition problem}, which is known to be \texttt{NP-hard}.

\begin{proposition}\label{prop:NPhard-MNL}
The MCP under the MNL model can be solved in \(\mathcal{O}(m \log m)\) if  $T=1$, and is \texttt{NP-hard} if $T\geq 2$.
\end{proposition}

To the best of our knowledge, Proposition \ref{prop:NPhard-MNL} marks the first time such a hardness result has been explored in the context of the MCP under the MNL model. To situate this result within the broader context of choice-based optimization, \cite{rusmevichientong2014assortment} demonstrates that the assortment optimization problem under the MMNL is NP-hard, even with only two customer classes. It is important to note that the objective function in \cite{rusmevichientong2014assortment} is more complex than the one in \eqref{prob:CNL-MCP} under the same choice model \footnote{The objective function of the assortment problem under MMNL becomes an MCP objective function under MNL with multiple demand points when all prices are the same over products.}. This implies that the assortment problem considered in their work is more challenging than \eqref{prob:CNL-MCP} (under the same choice model and number of customer classes). Consequently, our NP-hardness result, as stated in Proposition \ref{prop:NPhard-MNL}, is stronger than the one established for the assortment problem in \cite{rusmevichientong2014assortment}.

Proposition \ref{prop:NPhard-MNL} indicates that the MCP is generally hard to solve (unless \texttt{P=NP}) when there is more than one customer type. We now examine the hardness of the problem when there is only one customer type. Under the NL model - a more general model compared to the MNL model, where the IIA property is relaxed by assigning facilities to different nests, but each facility cannot belong to more than one nest, Proposition \ref{prop:NL-poly} below shows that the MCP under the NL model is still polynomial-time solvable.

\begin{proposition}\label{prop:NL-poly}
If there is only one customer type \((T=1)\) and given any \(\epsilon > 0\), the MCP under the NL model can be solved to an \(\epsilon\)-optimal solution in \(\mathcal{O}(m^2 r^2 \log(1/\epsilon))\). In other words, a solution $\widehat{\bx}$ such that $F(\widehat{\bx}) \geq \max_{\bx\in \cX} F(\bx) -\epsilon$ can be found in \(\mathcal{O}(m^2 r^2 \log(1/\epsilon))\).
\end{proposition}
The proof can be achieved by converting the MCP into a sequence of subproblems using the Dinkelbach transform \citep{dinkelbach1967nonlinear}. We then demonstrate that each subproblem can be solved exactly in polynomial time using dynamic programming.
Proposition \ref{prop:NL-poly} implies that if the choice model is NL and there is only one customer type, we can solve \eqref{prob:CNL-MCP} to \(\epsilon\)-optimality in polynomial time, where the complexity is proportional to \(\log(1/\epsilon)\). One can select \(\epsilon\) sufficiently small, for example, such that \(\epsilon < F^* - \max_{\bx \in \cX, F(\bx) < F^*} F(\bx)\), where \(F^*\) is the optimal value of \eqref{prob:CNL-MCP}. Then, it is guaranteed that the procedure described in the proof of Proposition \ref{prop:NL-poly} always returns optimal solutions to \eqref{prob:CNL-MCP}.

We are now ready to state our main NP-hardness result. Theorem \ref{th:NP-hard} below shows that the MCP under the CNL model is NP-hard even when there is only  $1$ customer class and $2$ nests. The proof can be established by considering an instance of the MCP with \(\sigma_{tn} \approx 0\) for all \(n\) and connecting the MCP to the set partition problem, which is known to be NP-hard. 

\begin{theorem}\label{th:NP-hard}
The MCP under the CNL model in \eqref{prob:CNL-MCP} is NP-hard even when there is only one customer type (i.e., \(T=1\)) and $2$ nests (\(N=2\)).
\end{theorem}
\begin{proof}
The NP-hardness can be verified by converting \eqref{prob:CNL-MCP} with two nests into an instance of the MNL-based problem with two customer types. We first write the objective function under the CNL choice model as
\[
F(\bx) = q_1 \frac{\sum_{n \in [N]} W_{1n}^{\sigma_{1n} - 1} \left( \sum_{i \in [m]} \alpha^1_{in} e^{v_{1i}/\sigma_{1n}} x_i \right)}{\sum_{n \in [N]} W_{1n}^{\sigma_{1n}}}
\]
We can choose \(\sigma_{1n} \approx 0\) for all \(n \in [2]\) and \(v_{1i} = 0\) for all \(i \in [m] \cup \cC\). Numerically, this can be done by selecting \(\sigma_{1n}\) close to zero and \(v_{1i}\) much smaller than \(\sigma_{1n}\). With this selection, we have \(e^{v_{1i}/\sigma_{1n}} = 1\) for all \(n \in [N]\) and \(i \in [m] \cup \cC\), and \(W_{1n}^{\sigma_{1n}} = 1\). We can then write the objective function as
\[
F(\bx) = q_1 \sum_{n \in [N]} \frac{\left( \sum_{i \in [m]} \alpha^1_{in} x_i \right)}{ U^c_n + \sum_{i \in [m]} \alpha^1_{in} x_i }
\]
where \(U^c_n = \sum_{i \in \cC} \alpha^1_{in}\). We now select \(v'_{ni} = \log \alpha^1_{in}\) for all \(n \in [N]\) and \(i \in [m] \cup \cC\). We continue to write the objective function as:
\[
F(\bx) = \sum_{n \in [N]} q_1 \frac{\left(\sum_{i \in [m]} e^{v'_{ni}} x_i \right)}{ U^c_n + \sum_{i \in [m]} e^{v'_{ni}} x_i }
\]
which is similar to the objective function under MNL with two customer types, where each nest now corresponds to a customer type. From the proof of Proposition \ref{prop:NPhard-MNL}, we know that this problem is NP-hard by connecting it with the Set Partition problem. This confirms the NP-hardness of \eqref{prob:CNL-MCP} with \(T=1\) and \(N=2\).
\end{proof}

The NP-hardness stated in Theorem \ref{th:NP-hard} is tight, in the sense that relaxing any of its conditions generally leads to a polynomial-time solvable problem. Specifically, if the choice model is simpler than the CNL (such as MNL or NL), then from Propositions \ref{prop:NPhard-MNL} and \ref{prop:NL-poly}, we know that the problem is polynomial-time solvable. Additionally, if there are fewer than $2$ nests, i.e., \(N=1\), then it can be seen that the objective function becomes:
\[
F(\mathbf{x}) = \sum_{t \in [T]} \frac{\sum_{n \in [N]} W_{tn}^{\sigma_{tn} - 1} \left( \sum_{i \in [m]} \alpha^t_{in} V_{tin} x_i \right)}{\sum_{n \in [N]} W_{tn}^{\sigma_{tn}}} = \sum_{t \in [T]} q_t \frac{\sum_{i \in [m]} V_{ti1} x_i}{\sum_{j \in [m] \cup \mathcal{C}} V_{tj1} x_j}
\]
which is the objective function of the MCP under the MNL model, and we know that when \(T=1\), the problem can be solved in \(\mathcal{O}(m \log m)\) (Proposition \ref{prop:NPhard-MNL}).

\section{Solution Methods}\label{sec:methods}
In this section, we discuss an exact solution method to solve the challenging problem in \eqref{prob:CNL-MCP}. Previous works have primarily focused on the \CSE  setting (where the competitor's facilities are separated in the cross-nested structure). Here, we will examine both \CSE  and the general \CSH  settings. Our main findings below show that while the objective function under \CSE  is concave in $\bx$, the objective function under \CSH is not. Furthermore, we show that, through some variable changes, the MCP under {\CSH} can be converted into an equivalent \textit{mixed-integer exponential cone convex program}, allowing it to be solved exactly via Cutting Plane or B\&C methods.

\subsection{Concavity}
We first examine the concavity of the objective function $F(\bx)$ under both  \CSE     and \CSH configurations.
In Proposition \ref{pro:incr-concv} below, we state that the objective function of \eqref{prob:CNL-MCP} is generally concave under \CSE:
\begin{proposition}\label{pro:incr-concv}
Under \CSE (i.e., the competitor's facilities are not sharing the cross-nested structure with the opening facilities), $F(\bx)$ in \eqref{prob:CNL-MCP} is monotonically increasing and concave in $\bx$.     
\end{proposition}
\begin{proof}
To see why $F(\bx)$ is monotonic and concave in $\bx$, we consider the formulation in \eqref{eq:obj-CSE}:
\begin{equation}
F(\mathbf{x})= \sum_{t \in [T]} q_t - \sum_{t \in [T]} \frac{U^c_t}{U^c_t + \sum_{n \in [N]} W_{tn}^{\sigma_{tn}}}
\end{equation}
and observe that each $W_{tn} = \sum_{i\in[m]}\alpha^t_{in}V_{tin}x_i + \sum_{i\in[m]}\alpha^t_{in}V_{tin}$ is linear in $\bx$ and monotonically increasing in each $x_i$, $i\in [m]$. Therefore, $W_{tn}^{\sigma_{tn}}$ is concave and monotonically increasing in $\bx$ (since $\sigma_{tn}\leq 1$). We then have that $U^c_t + \sum_{n \in [N]} W_{tn}^{\sigma_{tn}}$ is concave and monotonically increasing in $\bx$. We now see that the function $f(z) = 1/z$ is convex and decreasing in $z\in\bbR_+$, thus each fraction 
\[
\frac{U^c_t}{U^c_t + \sum_{n \in [N]} W_{tn}^{\sigma_{tn}}}
\]
is convex and decreasing in $\bx$, which implies that $F(\bx)$ is concave and increasing in $\bx$. 
\end{proof}

\citep{mendez2023follower} prove that the objective function under the NL model and \CSE setting is concave. The result stated in Proposition \ref{pro:incr-concv} generalizes this result by claiming that the objective function is also concave under the CNL model. In general, concavity is possible because the objective function can be written as a sum of fractions with constant numerators. If this is not the case, as under the \CSH setting, the objective function will no longer be convex. We state this result in Proposition \ref{prop:non-concave} below:
\begin{proposition}\label{prop:non-concave}
    Under \CSH (i.e., the competitor's facilities are sharing the cross-nested structure with the opening facilities), $F(\bx)$ in \eqref{prob:CNL-MCP} is monotonically increasing, but \textbf{not concave} in $\bx$.
\end{proposition}
While it is challenging to provide a counterexample with one or two variables to demonstrate that the objective function \( F(\bx) \) is not concave in \( \bx \) (although it can be shown that \( F(\bx) \) is concave in each individual element \( x_i \), for any \( i \in [m] \)), our experiments reveal instances where an outer-approximation method with cuts generated directly based on \( F(\bx) \) fails to return optimal solutions. This directly implies the non-concavity of \( F(\bx) \).

\subsection{Convexification}

The non-concavity of \( F(\bx) \) implies that \eqref{prob:CNL-MCP} cannot be solved directly to optimality using outer-approximation cuts. Fortunately, we can demonstrate below that, with some changes of variables, we can transform \eqref{prob:CNL-MCP} into a concave program. To begin, let us rewrite the objective function of \eqref{prob:CNL-MCP} as
\begin{align}
    F(\bx)  &= \sum_{t\in [T]}q_t\frac{\sum_{n\in [N]}{W_{tn}}^{\sigma_{tn}-1} (\sum_{i\in [m]}\alpha^t_{in}V_{tin} x_i)}{ \sum_{n\in [N]} W_{tn}^{ \sigma_{tn}}}\nonumber\\
    &=\sum_{t\in [T]}q_t\frac{\sum_{n\in [N]}{W_{tn}}^{\sigma_{tn}-1} (W_{tn}  - U^c_{tn})}{ \sum_{n\in [N]} W_{tn}^{ \sigma_{tn}}}\nonumber \\
    &= \sum_{t\in [T]}q_t  - \sum_{t\in [T]}q_t \frac{\sum_{n\in [N]}{W_{tn}}^{\sigma_{tn}-1}  U^c_{tn}}{ \sum_{n\in [N]} W_{tn}^{ \sigma_{tn}}},\nonumber 
\end{align}
where $U^c_{tn} = \sum_{i\in \cC \cap \cN_n^t} \alpha^t_{in}V_{tin}$. We then define some new variables as follows:
\allowdisplaybreaks
\begin{align}
   y_{tn} &= \log(W_{tn}^{\sigma_{tn}-1} U^c_{tn}),\forall t\in [T], n\in [N]\nonumber\\
   z_{t} &= \log\left(\sum_{n\in [N]} W_{tn}^{\sigma_{tn}}\right)
 \end{align}
 We then write \eqref{prob:CNL-MCP} as  the following equivalent nonlinear program:
 \allowdisplaybreaks
\begin{align}
    \max_{\textbf{\bx,\by,\bz}}\quad\quad& 
    \sum_{t\in [T]}q_t  - \sum_{t\in [T]}\sum_{n\in [N]}q_t \exp\left( y_{tn}-z_{t}\right)\label{prob:CNL-non-convex}\\
    \text{s.t.}\quad 
    &  y_{tn} = \log(W_{tn}^{\sigma_{tn}-1} U^c_{tn}),\forall t\in [T], n\in [N]\label{eq:ctr-n1}\\
  & z_{t} =   \log\left(\sum_{n\in [N]} W_{tn}^{\sigma_{tn}}\right),~ \forall t\in [T]\label{eq:ctr-n2}\\
  & W_{tn} = \sum_{i\in [m]} \alpha^t_{in} V_{tin} x_i + U^c_{tn},~\forall t\in [T], n\in [N] \nonumber\\
  & \sum_{i\in [m]}x_i = r\nonumber\\
  & \bx \in \{0,1\}^m, \by \in \bbR^{T\times N},~\bz \in \bbR^{T}\nonumber.
\end{align}
It can be further shown that the above nonlinear non-convex program can formulated as the following optimization problem with linear objective and convex constraints.

\begin{theorem}\label{thrm:CNL-convex}
    The MCP under the CNL model in \eqref{prob:CNL-MCP} is equivalent to the following mixed-integer nonlinear convex program:
    \allowdisplaybreaks
    \begin{align}
    \max_{{\bx,\by,\bz,\btheta,\bW}}\quad\quad& 
    \sum_{t\in [T]}q_t  - \sum_{t\in [T]}\sum_{n\in [N]}q_t \theta_{tn} \label{prob:CNL-Convex-refor}\tag{\sf MCP-CNL-2}\\
    \text{s.t.} \quad\quad
    &   \theta_{tn} \geq \exp\left( y_{tn}-z_{t}\right),~\forall t\in [T], n\in[N]\label{eq:ctr-1} \\
    & y_{tn} \geq (\sigma_{tn}-1)\log(W_{tn})  + \log(U^c_{tn}),\forall t\in [T], n\in [N]\label{eq:ctr-2}\\
  & z_{t} \leq \log\left(\sum_{n\in [N]} W_{tn}^{\sigma_{tn}}\right),~ \forall t\in [T]\label{eq:ctr-3}\\
  & W_{tn} = \sum_{i\in [m]} \alpha^t_{in} V_{tin} x_i + U^c_{tn},~\forall t\in [T], n\in [N] \nonumber\\
  & \sum_{i\in [m]}x_i = r\nonumber\\
  & \bx \in \{0,1\}^m, \by,\bW,\btheta \in \bbR^{T\times N},~\bz \in \bbR^{T}\nonumber.
\end{align}
    
\end{theorem}

\begin{proof}
From the non-convex problem in \eqref{prob:CNL-non-convex}, we can let \(\theta_{tn} = \exp(y_{tn} - z_{t})\). Since we are solving a maximization problem, each \(\theta_{tn}\) will need to be minimized. Thus, the equalities \(\theta_{tn} = \exp(y_{tn} - z_{t})\) can be safely converted to the convex constraints \(\theta_{tn} \geq \exp(y_{tn} - z_{t})\). Similarly, since we want to minimize each component \(\exp(y_{tn} - z_t)\), for any \(t \in [T]\), \(n \in [N]\), each additional variable \(y_{tn}\) needs to be minimized, and each \(z_t\) needs to be maximized as much as possible. As a result, the constraints \eqref{eq:ctr-n1} and \eqref{eq:ctr-n2} can be safely converted to inequality constraints \eqref{eq:ctr-2} and \eqref{eq:ctr-3}, respectively.

To see that all the constraints in \eqref{eq:ctr-1} are convex, we note that the functions \(\exp(y_{tn} - z_t)\) are convex in \(\by, \bz\), thus Constraints \eqref{eq:ctr-1} are convex. Moreover, the function \(\log(W_{tn})\) is concave in \(W_{tn}\) and \(\sigma_{tn} \leq 1\), implying that Constraints \eqref{eq:ctr-2} are concave.

For Constraints \eqref{eq:ctr-3}, we observe that each function \(W_{tn}^{\sigma_{tn}}\) is concave in \(W_{tn}\) (because \(\sigma_{tn} \leq 1\)). Thus, \(\sum_{n \in [N]} W_{tn}^{\sigma_{tn}}\) is concave in \(\bW\) (where \(\bW\) is a vector containing all \(W_{tn}\) for any \(t \in [T]\) and all \(n \in [N]\)). For ease of notation, let \(\psi_t(\bW) = \sum_{n \in [N]} W_{tn}^{\sigma_{tn}}\). We now prove that the function \(\cT(\bW) = \log(\psi_t(\bW))\) is concave in \(\bW\). This can be verified by noting that the function \(h(t) = \log(t)\) is concave in \(t\). Thus, for any \(\alpha \in [0, 1]\) and two vectors \(\bW^1, \bW^2\), this convexity implies:
\begin{align}
    \alpha \log(\psi_t(\bW^1)) + (1-\alpha) \log(\psi_t(\bW^2)) &\leq \log\left(\alpha \psi_t(\bW^1) + (1-\alpha) \psi_t(\bW^2)\right)\nonumber\\
    &\stackrel{(a)}{\leq} \log\left(\psi_t(\alpha \bW^1 + (1-\alpha) \bW^2)\right)\label{eq:1234}
\end{align}
where \((a)\) holds because \(\alpha \psi_t(\bW^1) + (1-\alpha) \psi_t(\bW^2) \leq \psi_t(\alpha \bW^1 + (1-\alpha) \bW^2)\) (\(\psi_t(\bW)\) is concave in \(\bW\)) and the function \(\log(t)\) is increasing in \(t\). The inequality in \eqref{eq:1234} confirms the concavity of \(\log(\psi_t(\bW))\). Thus, Constraints \eqref{eq:ctr-3} are convex. This completes the proof.
\end{proof}

Proposition \ref{prop:non-concave} indicates that the objective function under \CSH is generally non-concave, rendering a direct outer-approximation approach \citep{Ljubic2018outer, mai2020multicut} no longer an exact method. However, Theorem \ref{thrm:CNL-convex} demonstrates that \eqref{prob:CNL-MCP} can be converted into a mixed-integer convex program, where an outer-approximation method can be used to solve it to optimality. In the subsequent section, we discuss how to solve \eqref{prob:CNL-Convex-refor} using CP or B\&C procedures, incorporating outer-approximation and submodular cuts.

\subsection{Cutting Plane and B\&C Approaches}\label{sec:CP-BC}

The outer-approximation approach \citep{duran1986outer,OA_Bonami2008algorithmic} provides a powerful and widely used technique for solving mixed-integer nonlinear problems, particularly those with convex objectives and constraints. The general idea is to define a master problem, typically in the form of a MILP, where the nonlinear constraints and objective function are replaced by a sequence of outer-approximation cuts. By iteratively adding cuts and solving the master problem, the method can guarantee convergence to an optimal solution after a finite number of iterations. These valid cuts and the master problem can be incorporated into a CP or B\&C procedure.

In the context of the MCP under the CNL model, Theorem \ref{thrm:CNL-convex} demonstrates that the original MINLP can be transformed into a mixed-integer convex program, as given in \eqref{prob:CNL-Convex-refor}, enabling the application of an outer-approximation algorithm. To describe the method, and for ease of notation, let us define the nonlinear functions in \eqref{prob:CNL-Convex-refor} as follows:
\allowdisplaybreaks
\begin{align}
   \Psi^{\theta}_{tn}(\by,\bz)& =  \exp\left( y_{tn}-z_{t}\right),~\forall t\in [T], n\in[N] \nonumber\\
   \Psi^y_{tn}(\bW) &=  (\sigma_{tn}-1)\log(W_{tn})  + \log(U^c_{tn}),\forall t\in [T], n\in [N]\nonumber\\
   \Psi^z_{tn}(\bW) &= \log\left(\sum_{n\in [N]} W_{tn}^{\sigma_{tn}}\right),~ \forall t\in [T].\nonumber
\end{align}
We know that $\Psi^\theta_{tn}(\by,\bz)$ and $\Psi^y_{tn}(\bW)$ are convex in their inputs, while $\Psi^z_{tn}(\bW)$ is concave in $\bW$, thus the following outer-approximation cuts are valid for \eqref{prob:CNL-Convex-refor}:
\allowdisplaybreaks
\begin{align}
   y_{tn}&\geq \nabla_{\bW}\Psi^{y}_{tn}(\Bar{\bW})^\transpose(\bW-\Bar{\bW}) + \Psi^y_{tn}(\Bar{\bW}),~\forall t\in [T], n\in [N] \label{oa-1}\tag{\sf OA1}\\
   z_{t}&\leq \nabla_{\bx}\Psi^z_{t}(\overline{\bW})^\transpose(\bW-\overline{\bW}) + \Psi^z_{t}(\overline{\bW}),~\forall t\in [T] \label{oa-2}\tag{\sf OA2}\\
   \theta_{tn} &\leq \left(\nabla_{\by}\Psi^\theta_{tn}(\Bar{\by},\Bar{\bz})\right)^\transpose(\by-\Bar{\by}) + \left(\nabla_{\bz}\Psi^\theta_{tn}(\Bar{\by},\Bar{\bz})\right)^\transpose(\bz-\Bar{\bz})\nonumber\\
   &\qquad\qquad + \Psi^\theta_{tn}(\oby,\obz),~\forall t\in [T], n\in [N] \label{oa-3}.\tag{\sf OA3}
\end{align}
Furthermore, prior studies demonstrate that the objective function \( F(\bx) \), if defined as a set function, is submodular. Specifically, if we define
\begin{align}
    \Phi_t(S) &= q_t\left(1 - \frac{\sum_{n\in [N]} W_{tn}(S)^{\sigma_{tn}-1}  U^c_{tn}}{\sum_{n\in [N]} W_{tn}(S)^{\sigma_{tn}}}\right), \forall t \in [T]\nonumber\\
    \Phi_t(\bx) &= q_t\left(1 - \frac{\sum_{n\in [N]} W_{tn}(\bx)^{\sigma_{tn}-1}  U^c_{tn}}{\sum_{n\in [N]} W_{tn}(\bx)^{\sigma_{tn}}}\right), \forall t \in [T]\nonumber
\end{align}
where \( W_{tn}(\bx) =  \sum_{i \in S} \alpha^t_{in} V_{tin} x_i + U^c_{tn} \) and \( W_{tn}(S) =  \sum_{i \in [m]} \alpha^t_{in} V_{tin} + U^c_{tn} \). It can be shown that \( \Phi_t(S) \) is monotonically increasing and submodular in \( S \)  \citep{dam2022submodularity,berbeglia2020assortment}. This suggests that submodular cuts can be incorporated to enhance the CP and B\&C procedures. To this end, if we denote \( \phi_t = q_t - \sum_{n \in [N]} q_t \theta_{tn} \) in \eqref{prob:CNL-Convex-refor}, then the  inequalities \( \phi_t \geq \Phi_t(\bx) \) are valid for \eqref{prob:CNL-Convex-refor}. 

To introduce submodular cuts for \eqref{prob:CNL-Convex-refor}, for each \( t \in [T] \) and \( i \in [m] \), we define \(\rho_{ti}(\bx) = \Phi_t(\bx + \textbf{e}_i) - \Phi_t(\bx)\), where \(\textbf{e}_i\) is a binary vector of size \( m \) with zero elements except for the \(i\)-th element, and \(\bx + \textbf{e}_i\) denotes the binary vector such that each element is the maximum of the corresponding elements in \(\bx\) and \(\textbf{e}_i\). The functions \(\rho_{ti}(S)\) are often referred to as marginal gains, representing the gains from adding an item \( k \) to the set \( S \). According to the submodularity of \( \Phi_t(S) \) for all \( t \in [T] \), the following submodular cuts are also valid for \eqref{prob:CNL-Convex-refor}:
\allowdisplaybreaks
\begin{align}
    \phi_{t} &\leq \sum_{i\in [m]}\rho_{ti}(\obx)(1-\obx_i)x_i - \sum_{i\in [m]}\rho_{ti}(\textbf{e}-\textbf{e}_i)\obx_i(1-x_i) + \Phi_t(\obx), \forall t\in [T]\label{sc-1}\tag{\sf SC1}\\
    \phi_{t} &\leq \sum_{i\in [m]}\rho_{ti}(\textbf{0}) - \sum_{i\in [m]}\rho_{ti}(\obx-\textbf{e}_i)\obx_i(1-x_i) + \Phi_t(\obx), \forall t\in [T]\label{sc-2}\tag{\sf SC2}
\end{align}
In summary, we employ the following master problem: 
\allowdisplaybreaks
  \begin{align}
    \max_{{\bx,\by,\bz,\bW,\btheta,\pmb{\phi}}}\quad\quad& 
    \sum_{t\in [T]}\phi_t\label{prob:CNL-master}\tag{\sf Master}\\
    \text{s.t.} \quad\quad
    & \phi_t \geq  q_t - \sum_{n \in [N]} q_t \theta_{tn}  \nonumber \\
  & W_{tn} = \sum_{i\in [m]} \alpha^t_{in} V_{tin} x_i + U^c_{tn},~\forall t\in [T], n\in [N] \nonumber\\
  & \sum_{i\in [m]}x_i = r\nonumber\\
  & \texttt{[Outer-approximation cuts \eqref{oa-1}, \eqref{oa-2},\eqref{oa-3}]}\nonumber \\
  & \texttt{[Submodular cuts \eqref{sc-1}, \eqref{sc-2}]} \nonumber\\
  & \bx \in \{0,1\}^m, \by,\bW ,\btheta \in \bbR^{T\times N},~\bz,\pmb{\phi} \in \bbR^{T}\nonumber.
\end{align}
A CP method works as follows. Starting from the master problem \eqref{prob:CNL-master} without any outer-approximation or submodular cuts, at each iteration, we solve the master problem to obtain a solution candidate and add outer-approximation cuts and submodular cuts to the master problem. This process terminates when we find a solution candidate \((\bx^*,\by^*,\bz^*,\bW^*,\pmb{\phi}^*)\) such that \(\sum_{t} \phi^*_t = F(\bx^*)\). Since the master problem also yields objective values that upper bound the optimal value of the original problem \eqref{prob:CNL-Convex-refor}, the stopping condition guarantees that the CP, when terminated, will always return an optimal solution to \eqref{prob:CNL-Convex-refor}. Moreover, it is known that the above CP procedure will never produce the same solution candidate throughout its process, so it will always terminate after a finite number of iterations \citep{duran1986outer}.

Algorithm \ref{algo:CP} describes the main steps of our CP algorithm. It maintains an upper bound (UB) and a lower bound (LB) of the optimal value of the MCP. After each iteration, the master problem \eqref{prob:CNL-master} is solved to obtain a new solution candidate. We then add outer-approximation cuts or submodular cuts for any violated constraint and update the LB as the objective value of the best location solution found, and the UB as the objective value given by the master problem. Since the master problem is a relaxation of \eqref{prob:CNL-Convex-refor}, it is guaranteed that the UB is always greater than \(F(\bx)\) for any feasible solution \(\bx\). Thus, when the stopping condition is met, it is guaranteed that the solution candidate obtained in the last step is optimal for the MCP. It is also guaranteed that this process always terminates after a finite number of iterations \citep{duran1986outer,bui2022cuttingplanealgorithmsnonlinear,OA_Bonami2008algorithmic}. 
\begin{algorithm}[!ht]            
        \caption{Cutting Plane Algorithm}
         \label{algo:CP} 
        \DontPrintSemicolon
        Choose a small $\epsilon \geq 0$ as the optimal gap.\;
         Initialize the master problem \eqref{prob:CNL-master} \\
         \KwIn{A candidate solution $({\overline{\bx},\overline{\by},\overline{\bz},\overline{\bW},\overline{\btheta},\overline{\pmb{\phi}}})$ of the master problem.}
        $\text{UB} \gets +\infty, \text{LB} \gets -\infty$.\;
        \While{$\text{UB} - \text{LB} > \epsilon$}{
            Solve \eqref{prob:CNL-master} to obtain new candidate solution $({\overline{\bx},\overline{\by},\overline{\bz},\overline{\bW},\overline{\btheta},\overline{\pmb{\phi}}})$.\;
            \comments{Add outer-approximation cuts for violated constraints}\;
             If any nonlinear constraints \eqref{eq:ctr-1}, \eqref{eq:ctr-2} or \eqref{eq:ctr-3} are violated, then add the corresponding outer-approximation cuts as in \eqref{oa-1}
, \eqref{oa-2} or \eqref{oa-3} \;
\comments{Add submodular cuts for violated constraints}\;
If $\overline{\phi}_t <\Phi_t(\overline{\bx})$, then add submodular cuts \eqref{sc-1} and \eqref{sc-2}, for any $t\in [T]$\;
\comments{Update the lower  and upper bounds}\;
                $\text{UB} \gets \sum_{t\in [T]} \phi_t, \text{LB} \gets \max \{\text{LB}, {F}(\overline{\bx})\}$.\;
          
        }
        \Return $\overline{\bx}$.\;
\end{algorithm}

Similar to prior work, the outer-approximation and submodular cuts can also be incorporated into a B\&C algorithm to solve the mixed-integer convex problem \citep{Ljubic2018outer,pham2024competitive}. Specifically, we maintain the master problem \eqref{prob:CNL-master} during the procedure. At each iteration, we check for violating constraints and add outer-approximation cuts as in \eqref{oa-1}-\eqref{oa-2}, \eqref{oa-3}, or submodular cuts as in \eqref{sc-1} and \eqref{sc-2} to the master problem using the \texttt{lazy-cut callback} procedure within solvers such as Gurobi and CPLEX.

\section{Experiments}\label{sec:experiments}
In this section, we present experimental results to evaluate the performance of our proposed methods for solving the MCP under the CNL model, as well as numerical results to compare across choice models. Additional experiments are included in the appendix to (i) illustrate the impact of various CNL parameters on the performance of our algorithm, and (ii) compare the use of outer-approximation and submodular cuts within the CP or B\&C procedures.

\subsection{Experimental Settings}
We utilize three benchmark datasets commonly used in previous MCP studies \citep{Ljubic2018outer,mai2020multicut,dam2022submodularity}:
\begin{itemize}
    \item \textbf{ORlib}: This dataset comprises 11 problem instances, including four instances with \((T,m) = (50,25)\), four instances with \((T,m) = (50,50)\), and three instances with \((T,m) = (1000,100)\). Instances from this dataset are labeled with the prefix ``\tt{cap}''.
    \item \textbf{HM14}: This dataset consists of 15 problem instances with \(T \in \{50, 100, 200, 400, 800\}\) and \(m \in \{25, 50, 100\}\). Instances from this dataset are labeled with the prefix ``\tt{OUR}''.
    \item \textbf{NYC}: This dataset comes from a large-scale park-and-ride location problem in New York City, featuring \(T = 82341\) and \(m = 59\). Given the substantial number of demand points, solving the MCP-CNL within a limited time budget is very challenging. Therefore, we use a subset of the demand points (\(T = 2000\)) for the evaluation while maintaining the same set of locations.
\end{itemize}

The maximum number of facilities to be opened, \( r \), ranges from 2 to 10. The deterministic part of the utility is calculated as \( v_{ti} = -\beta c_{ti} \) for a location \( i \in S \), and \( v_{ti'} = -\beta \alpha c_{ti'} \) for any competitor \( i' \in \cC \), where \( c_{ti'} \) is the cost between customer type \( t \in [T] \) and location \( i' \in \cC \). The parameter \(\alpha\), representing the competitiveness of competitors, is chosen from \(\{0.5, 1, 2\}\). The parameter \(\beta\), representing customer sensitivity to perceived utilities, is chosen from \(\{0.01, 0.05, 0.1\}\) for the \textbf{ORlib} dataset and from \(\{0.001, 0.005, 0.01\}\) for the \textbf{NYC} and \textbf{HM14} datasets. The \(\beta\) value for the \textbf{NYC} and \textbf{HM14} datasets is chosen to be smaller due to the higher cost between zones and locations compared to the \textbf{ORlib} set. Consequently, each problem has 81 different instances.

The CNL parameters are randomly generated as follows. The dissimilarity parameter \(\sigma_{tn}\) for nest \(\mathcal{N}_{n}^{t}\) is drawn from a normal distribution \(\mathcal{N}(\mu, \omega^2)\) with a mean \(\mu\) of 0.5 and a standard deviation \(\omega\) of 0.2. Values outside the interval \([0.1, 1]\) are clipped to the interval edges to ensure they remain within a reasonable range. Moreover, the number of nests, \(N\), is set to 5. Since the nests may overlap, we introduce a parameter \(\gamma \geq 1\), which represents the average number of nests to which a single location belongs. This parameter, referred to as the overlapping rate, directly influences how locations are shared among the different nests and will be used to control the degree of overlap across nests. 

In this experiment, the value of \(\gamma\) is set to $1.2$. This parameter influences the cross-nested correlation structure across \(m\) locations and \(N\) nests, which is constructed as follows. We begin by randomly assigning the \(m\) locations to the \(N\) nests. Next, we randomly sample \(\lceil (\gamma - 1) m \rceil\) locations from the \(m\) locations. These sampled locations are then randomly assigned to the \(N\) nests, ensuring that each location is assigned no more than once per nest and that each nest contains at least two distinct locations. The allocation parameter \(\alpha_{in}^t\) for \(i \in [m] \cup \cC\) is randomly generated from a uniform distribution over the interval \([0,1]\). These parameters are normalized such that \(\alpha_{in}^{t} = 0\) if location \(i\) is not a member of nest \(\mathcal{N}_{n}^t\) and \(\sum_{n \in [N]} \alpha_{in}^{t} = 1\). For simplicity, we assume that all customer types share the same set of nests, i.e., \(\mathcal{N}_{n}^1 = \mathcal{N}_{n}^2 = \ldots = \mathcal{N}_{n}^T\) for all \(n \in [N]\). However, we maintain differences in the allocation parameters across all customer types, allowing for variation in how different customer types perceive the utility of each location within the nests.


In this experiment, we will follow the more general configuration \CSH, where the competitors' facilities and the available locations \([m]\) share a common nested structure. As discussed, this approach allows for a more comprehensive and accurate correlation structure between available facilities in the market, compared to the \CSE configuration.

We will compare our approaches, specifically {\textbf{CP}} and {\textbf{B\&C}}, in solving the mixed-integer convex program in \eqref{prob:CNL-Convex-refor} using outer-approximation and submodular cuts. For comparison, we will also include a Direct Outer-Approximation method (\textbf{DOA}), which solves the non-convex problem in \eqref{prob:CNL-MCP} using the standard outer-approximation scheme employed in previous work \citep{mai2020multicut}. Specifically, we formulate \eqref{prob:CNL-MCP} as:
\[
\max_{\bx \in \{0,1\}^m}~~\{ \theta \mid \theta \leq F(\bx)\}
\]
and use the following outer-approximation cuts of the form \(\theta \leq \nabla_\bx F(\overline{\bx})^\top(\bx-\overline{\bx}) + F(\overline{\bx})\) during a CP procedure. It is important to note that \(F(\bx)\) is not concave in \(\bx\) (Proposition \ref{prop:non-concave}), so the CP cannot guarantee optimal solutions and thus becomes heuristic.

We also include a Greedy Heuristic (\textbf{GH}) for comparison. This \textbf{GH} starts with an empty solution and iteratively selects locations one by one to maximize the objective value. The greedy procedure stops when the maximum cardinality is reached, $|S| = r$. To improve the greedy solution, we perform some exchange steps, swapping a chosen location with another one outside the current set. As discussed, since the objective function, as a set function, is monotonically increasing and submodular, the \textbf{GH} algorithm guarantees a \((1 - 1/e)\) approximation solution \citep{nemhauser1981maximizing}. Recent studies on MCP show that such a greedy and local search procedure can achieve state-of-the-art performance in terms of returning the best solutions within short computing times, compared to other approaches \citep{dam2022submodularity,dam2023robust}. By considering \textbf{GH}, we aim to compare the performance of this heuristic approach with our exact methods.

The algorithms were written in C++, and the MILP models were solved by IBM ILOG CPLEX $22.1.1$, where the number of CPUs is set to $8$ cores. All experiments were run on $13$th Gen Intel(R) Core(TM) i$5$-$13500$ @ $4.8$ GHz, and $64$ GB of RAM. The CP algorithm terminates when it reaches the maximum number of iterations $nb_{iter}=10000$ or the running time exceeds $3600$ seconds. The B\&C algorithm is also configured with a maximum runtime of 3600 seconds for each instance.

\subsection{Numerical Comparison}
We compare our algorithms (\textbf{CP} and \textbf{B\&C}), implemented with both outer-approximation and submodular cuts, against two baselines: \textbf{DOA} and \textbf{GH}, on \(81 \times 27\) instances. Comparisons between the two types of cuts used in \textbf{CP} and \textbf{B\&C} are provided in the appendix.

To evaluate the performance of our approaches compared to the baselines, we report the number of times each approach returns optimal solutions or the best objective values. It is important to note that \textbf{DOA} and \textbf{GH} are heuristics and cannot guarantee optimal solutions, whereas our \textbf{CP} and \textbf{B\&C} methods can ensure optimal solutions if they terminate within the time budget (i.e., 1 hour). We also report the average computing times for all approaches.




\begin{table}[htb]
    \centering
    \caption{Numerical results for the competitive facility location under CNL model, computing times are given in seconds.}
    \label{tab:main-results}
    \scalebox{0.8}{
    \begin{tabular}{ccccccccccccc}
    \hline
        \multirow{2}{*}{Instance set}& \multirow{2}{*}{$T$} & \multirow{2}{*}{$m$} & \multicolumn{2}{c}{\textbf{DOA}} & \multicolumn{2}{c}{\textbf{GH}} & \multicolumn{3}{c}{\textbf{B\&C}} & \multicolumn{3}{c}{\textbf{CP}} \\ \cmidrule(lr){4-5}\cmidrule(lr){6-7}\cmidrule(lr){8-10}\cmidrule(lr){11-13}
        ~ & ~ & ~ & \#Best & Time& \#Best & Time & \#Best & \#Opt & Time & \#Best & \#Opt & Time \\ \hline
        cap101 & 50 & 25 & \textbf{81} & 0.08 & 62 & 0.00 & \textbf{81} & \textbf{81} & 0.11 & \textbf{81} & \textbf{81} & 0.36 \\ \hline
        cap102 & 50 & 25 & \textbf{81} & 0.09 & 74 & 0.00 & \textbf{81} & \textbf{81} & 0.09 & \textbf{81} & \textbf{81} & 0.32 \\ \hline
        cap103 & 50 & 25 & 56 & 0.09 & 51 & 0.00 & \textbf{81} & \textbf{81} & 0.10 & \textbf{81} & \textbf{81} & 0.59 \\ \hline
        cap104 & 50 & 25 & 77 & 0.08 & \textbf{81} & 0.00 & \textbf{81} & \textbf{81} & 0.09 & \textbf{81} & \textbf{81} & 0.41 \\ \hline
        cap131 & 50 & 50 & 52 & 0.23 & 7 & 0.01 & \textbf{81} & \textbf{81} & 0.26 & \textbf{81} & \textbf{81} & 1.29 \\ \hline
        cap132 & 50 & 50 & 74 & 0.16 & 63 & 0.01 & \textbf{81} & \textbf{81} & 0.14 & \textbf{81} & \textbf{81} & 0.69 \\ \hline
        cap133 & 50 & 50 & 63 & 0.14 & 54 & 0.01 & \textbf{81} & \textbf{81} & 0.12 & \textbf{81} & \textbf{81} & 0.85 \\ \hline
        cap134 & 50 & 50 & 75 & 0.19 & 42 & 0.01 & \textbf{81} & \textbf{81} & 0.14 & \textbf{81} & \textbf{81} & 0.66 \\ \hline
        capa & 1000 & 100 & \textbf{81} & 8.15 & 30 & 0.82 & 80 & \textbf{76} & 1389.03 & \textbf{81} & \textbf{76} & 1615.53 \\ \hline
        capb & 1000 & 100 & \textbf{81} & 18.69 & 3 & 0.82 & \textbf{81} & \textbf{79} & 1262.18 & 75 & 60 & 1850.24 \\ \hline
        capc & 1000 & 100 & \textbf{81} & 105.64 & 0 & 0.82 & 71 & \textbf{52} & 2155.58 & 62 & 26 & 3099.91 \\ \hline\hline
        OUR & 50 & 25 & \textbf{81} & 0.24 & 58 & 0.00 & \textbf{81} & \textbf{81} & 0.16 & \textbf{81} & \textbf{81} & 0.92 \\ \hline
        OUR & 50 & 50 & 70 & 0.28 & 48 & 0.01 & \textbf{81} & \textbf{81} & 0.19 & \textbf{81} & \textbf{81} & 1.38 \\ \hline
        OUR & 50 & 100 & \textbf{81} & 0.56 & 31 & 0.04 & \textbf{81} & \textbf{81} & 0.34 & \textbf{81} & \textbf{81} & 2.54 \\ \hline
        OUR & 100 & 25 & \textbf{81} & 0.16 & 67 & 0.01 & \textbf{81} & \textbf{81} & 0.34 & \textbf{81} & \textbf{81} & 2.49 \\ \hline
        OUR & 100 & 50 & \textbf{81} & 0.48 & 50 & 0.02 & \textbf{81} & \textbf{81} & 0.48 & \textbf{81} & \textbf{81} & 4.45 \\ \hline
        OUR & 100 & 100 & \textbf{81} & 2.86 & 51 & 0.08 & \textbf{81} & \textbf{81} & 2.60 & \textbf{81} & \textbf{81} & 21.41 \\ \hline
        OUR & 200 & 25 & \textbf{81} & 0.78 & 53 & 0.02 & \textbf{81} & \textbf{81} & 1.99 & \textbf{81} & \textbf{81} & 19.08 \\ \hline
        OUR & 200 & 50 & 66 & 0.39 & 58 & 0.05 & \textbf{81} & \textbf{81} & 2.01 & \textbf{81} & \textbf{81} & 6.42 \\ \hline
        OUR & 200 & 100 & \textbf{81} & 13.03 & 4 & 0.17 & \textbf{81} & \textbf{81} & 15.47 & \textbf{81} & \textbf{81} & 226.64 \\ \hline
        OUR & 400 & 25 & \textbf{81} & 0.33 & 62 & 0.03 & \textbf{81} & \textbf{81} & 7.72 & \textbf{81} & \textbf{81} & 18.76 \\ \hline
        OUR & 400 & 50 & \textbf{81} & 0.61 & 60 & 0.10 & \textbf{81} & \textbf{81} & 6.89 & \textbf{81} & \textbf{81} & 43.06 \\ \hline
        OUR & 400 & 100 & \textbf{81} & 12.72 & 24 & 0.34 & \textbf{81} & \textbf{81} & 113.13 & \textbf{81} & 79 & 538.85 \\ \hline
        OUR & 800 & 25 & \textbf{81} & 0.38 & \textbf{81} & 0.06 & \textbf{81} & \textbf{81} & 75.86 & \textbf{81} & \textbf{81} & 143.43 \\ \hline
        OUR & 800 & 50 & \textbf{81} & 0.84 & 59 & 0.19 & \textbf{81} & \textbf{81} & 135.40 & \textbf{81} & \textbf{81} & 189.47 \\ \hline
        OUR & 800 & 100 & \textbf{81} & 17.00 & 58 & 0.66 & \textbf{81} & \textbf{79} & 1128.49 & 76 & 61 & 1977.20 \\ \hline \hline
        NYC & 2000 & 59 & \textbf{76} & 2.46 & 61 & 0.64 & \textbf{76} & \textbf{72} & 511.92 & \textbf{76} & 60 & 1678.24 \\ \hline        \hline
        Average & ~ & ~ & 76.56	& ~	& 47.85 & ~ & \textbf{80.41} & \textbf{79.26} & ~ & 79.70 & 76.41 & ~\\\hline
    \end{tabular}}
\end{table}

Table \ref{tab:main-results} presents our comparison results. Here, \#Best refers to the number of times each approach returns the best objective values (across all approaches), and \#Opt refers to the number of times each approach returns optimal solutions (applicable only for the exact methods \textbf{CP} and \textbf{B\&C}). The highest numbers of best and optimal solutions found over all approaches are highlighted in bold.
All computing times are reported in seconds. Each row of the table reports the averages for 81 instances, grouped by the number of customer types \(T\) and the number of locations \(m\). 

We first observe that, although \textbf{DOA} converges very quickly in most instances, it fails to return optimal solutions for many smaller instances. For example, for \texttt{cap103} instances, \textbf{DOA} terminates quickly but only returns optimal solutions for 56 out of 81 instances. This directly confirms the non-concavity stated in Proposition \ref{prop:non-concave}. 

The results in Table \ref{tab:main-results} also indicate that our approaches, \textbf{CP} and \textbf{B\&C}, outperform \textbf{DOA} and \textbf{GH} in terms of returning the best solutions. \textbf{CP} and \textbf{B\&C} return the best values for approximately 79.70 out of 81 and 80.41 out of 81 instances, respectively, across all instances. For small-sized instances with \(T \leq 400\), our approaches consistently return the best solutions in all cases. For larger-sized instances, \textbf{CP} and \textbf{B\&C} maintain strong performance. \textbf{B\&C} finds the best solutions in at least 71 out of 81 instances, while for \textbf{CP} there are 62 out of 81 instances (on average) are solved with the best solutions across seven settings: \(T = 800\) from \textbf{HM14}, \(T = 1000\) from \textbf{ORlib}, and \(T = 2000\) from \textbf{NYC}. The \textbf{GH} approach, while very fast as expected, performs poorly in terms of returning the best solutions, with an average of only 47.85 out of 81 instances yielding the best solutions.

Regarding optimality, we observe that most instances with \(T \leq 400\) are optimally solved by our exact methods, except for the \textbf{HM14} dataset with \(T = 400\) and \(m = 100\), where \textbf{CP} still has 2 instances not solved to optimality. However, for these instances, the optimal solution is still found. For larger values of \(T\), the proposed algorithms also perform well, solving at least 60 out of 81 instances to optimality for most larger instances.  Notably, the \texttt{capc} instances from \textbf{ORlib} appear to be the most challenging, as our exact methods solve only 52 out of 81 instances (for \textbf{B\&C}) and 26 out of 81 instances (for \textbf{CP}) to optimality.


Regarding the computing times, all instances with \(T \leq 400\) are solved to optimality very quickly, in less than 2 minutes for \textbf{B\&C} and less than 10 minutes for \textbf{CP}. Additionally, for instances where \(m \leq 50\), even large instances are solved mostly in under 3 minutes. The only exception is the \textbf{CP} algorithm, which takes an average of 189.47 seconds for the \((T, m) = (800, 50)\) instances. This highlights the impact of \(m\) on the computing time of the algorithms. Notably, when examining the average computing time for the \textbf{NYC} instances compared to the \tt{capa}, \tt{capb}, and \tt{capc} instances from the \textbf{ORlib} dataset, we observe that, despite having twice the number of customer types (2000 compared to 1000) but fewer locations (59 compared to 100), the average computation time on the \textbf{NYC} dataset is still significantly lower than that of the \tt{capa}, \tt{capb}, and especially \tt{capc} instances. The computing times for \textbf{GH} and \textbf{DOA} are significantly lower compared to our exact methods. \textbf{GH} requires less than a second to terminate, which is expected as it only involves a few steps of adding or exchanging locations. \textbf{DOA} also maintains a simple master problem with a much lower number of cuts added after each iteration, resulting in faster convergence times. However, it is worth noting that,  since the cuts added by \textbf{DOA} do not necessarily outer-approximate the objective function \(F(\bx)\), they may inadvertently cut off optimal solutions, causing the algorithm to stop at a local solution. This explains why \textbf{DOA}, despite its speed, does not consistently achieve optimal results.

\subsection{Comparison across Different Choice Models}

The use of the CNL model in decision-making offers the theoretical advantage of encompassing many other demand models as special cases (e.g., the MNL or NL) and being able to approximate many other general demand models, such as the MMNL or ranking-preference models \citep{fosgerau2013choice,LeC2024POM}. In practice, several prior studies have shown that the CNL model significantly outperforms other choice models in predicting people's demand \citep{beine2021new,ding2015cross}.
However, applying the CNL model in the context of the MCP comes with the cost of increased computational complexity, as demonstrated in the previous section. In this section, we aim to explore the extent of the performance loss incurred by simplifying the CNL model to simpler models such as the MNL or NL within the MCP context. To this end, we consider MCP-CNL instances and simplify the CNL model into its MNL or NL counterparts. We then solve these MNL-based and NL-based variants to obtain solutions. These solutions will be used to evaluate the loss associated with using such simplified models.

Specifically, we simplify the MCP-CNL instance to an MNL-based instance by setting all the dissimilarity parameters \(\pmb{\sigma}\) to $1$, i.e., \(\sigma_{tn} = 1\) for all \(t \in [T]\) and \(n \in [N]\). This results in the following MNL-based MCP to solve:
	\begin{align}
 		\max_{\bx}\qquad &\left\{\sum_{t\in [T]}q_t\frac{\sum_{i\in [m]}e^{v_{ti}} x_i}{ \sum_{i\in [m]}e^{v_{ti}} x_i + \sum_{i\in\cC} e^{v_{ti}}}\right\}\nonumber \\
        \text{subject to} \quad  & \sum_{i\in [m]} x_i \leq r\nonumber\\ 
        & \bx \in \{0,1\}^m\nonumber
	\end{align}

For the NL-based variant, since each location/facility can belong to only one nest, we need to assign each location/facility \(i \in [m]\) to the nest to which \(i\) has the highest membership value. Specifically, for each \(t \in [T]\) and \(i \in [m]\), we take \(n^* = \text{argmax}_{n \in [N]} \alpha^t_{in}\) and assign location \(i\) to nest \(n^*\) (i.e., set \(\alpha^t_{in} = 1\) if \(n = n^*\) and \(\alpha^t_{in} = 0\) otherwise). By doing this, we ensure each location \(i\) belongs to only one nest. We also retain the same dissimilarity parameters \(\pmb{\sigma}\) as in the CNL instances.

Each MNL-based or NL-based instance is then solved to optimality using our general exact method described in Section \ref{sec:methods}. Let \(\bx^{\MNL}\) and \(\bx^{\NL}\) be the optimal solutions to the MNL-based and NL-based instances, respectively. The percentage loss is then computed as follows:
 \[
 \tt{[\% Loss]} = \frac{F^* - F(\bx^{\MNL})}{F^*}\times 100\%; \text{ or }  \tt{[\% Loss]} = \frac{F^* - F(\bx^{\NL})}{F^*}\times100\%
 \]
where $F^*$ is the optimal value of the corresponding CNL instance. 
 
\begin{table}[htb]\footnotesize
\centering
\caption{Comparison results with other choice models.}
\begin{tabular}{ccccc}
\hline
        \multirow{2}{*}{Instance set} & \multirow{2}{*}{$T$} & \multirow{2}{*}{$m$} & \multicolumn{1}{c}{MNL} & \multicolumn{1}{c}{NL} \\ \cmidrule(lr){4-4}\cmidrule(lr){5-5}
        ~ & ~ & ~ & \%Loss & \%Loss \\ \hline
        cap101 & 50 & 25 & 4.18 & 7.28 \\ \hline
        cap102 & 50 & 25 & 4.83 & 7.77 \\ \hline
        cap103 & 50 & 25 & 1.68 & 5.65 \\ \hline
        cap104 & 50 & 25 & 3.19 & 6.57 \\ \hline
        cap131 & 50 & 50 & 3.50 & 7.75 \\ \hline
        cap132 & 50 & 50 & 4.77 & 8.14 \\ \hline
        cap133 & 50 & 50 & 4.59 & 8.01 \\ \hline
        cap134 & 50 & 50 & 3.54 & 8.36 \\ \hline\hline
        OUR & 50 & 25 & 3.67 & 6.30 \\ \hline
        OUR & 50 & 50 & 4.68 & 7.38 \\ \hline
        OUR & 50 & 100 & 5.86 & 8.09 \\ \hline
        OUR & 100 & 25 & 4.01 & 6.79 \\ \hline
        OUR & 100 & 50 & 5.49 & 8.17 \\ \hline
        OUR & 100 & 100 & 5.90 & 8.41 \\ \hline\hline
        Average & ~ & ~ & 4.28 & 7.48\\\hline
    \end{tabular}
\label{tab:comparison-models}
\end{table}

We conducted tests on small instances with \(T \leq 100\) from the \textbf{ORlib} and \textbf{HM14} datasets to ensure all instances are solved optimally. The percentage losses are presented in Table \ref{tab:comparison-models}, where each row shows average values for 81 instances. The results indicate significantly high percentage losses for all instances, with an average of 4.28\% for the MNL solutions and 7.48\% for the NL ones. Interestingly, the losses from the MNL solutions are remarkably higher than those from the NL model, suggesting that the simpler MNL model can achieve better solutions compared to the NL counterpart when the ground-truth model is the CNL. Additionally, due to its simpler structure, MNL-based variants are generally faster to solve compared to the NL and CNL-based MCP instances.



\section{Conclusion}\label{sec:conclusion}
In this paper, we studied a competitive facility location problem where customer behavior is forecasted using the CNL customer choice model. By incorporating the CNL, we have made a significant advancement in the literature on competitive facility location, as the CNL model is considered one of the most flexible and general choice models in the demand modeling literature. It encompasses many other choice models as special cases (e.g., the MNL or NL models) and can arbitrarily approximate many other general choice models (e.g., the MMNL or ranking-preference models). 

The facility location problem under the CNL is highly complex, and we demonstrated that it is \tt{NP-hard} even with only one customer class and two nests. We further delved into the problem structure and showed that when the competitor’s facilities are separated from opening facilities in the nested structure (i.e., \CSE), the objective function exhibits concavity. However, in a general configuration where the competitor’s facilities share a common nested structure with opening facilities, the objective function is not concave, making it challenging to solve to optimality. Interestingly, we show that the non-convex problem under \CSH can be equivalently converted into a mixed-integer convex program. We then derived outer-approximation and submodular cuts that can be incorporated into a CP or B\&C procedure to exactly solve the problem. Extensive experiments demonstrated the practicality and efficiency of our approaches in solving large-scale instances. We also provided numerical results to analyze the impact of the CNL parameters and the losses when simplifying the cross-nested correlation structure in the context of the MCP.

Future work will explore more advanced (but more complex) settings, such as the MCP under multi-level cross-nested models \citep{daly2006general,mai2017dynamic} or a joint facility and price optimization under the CNL model or other advanced choice models.

\bibliographystyle{plainnat_custom}
\bibliography{refs}

\pagebreak

\appendix

\section*{Appendix}

In this appendix, we include proofs that were not presented in the main paper. Additionally, we provide further numerical results to assess the impact of the CNL parameters on the performance of our algorithms. We also present comparison results for the outer-approximation and submodular cuts when being incorporated into the CP and B\&C procedures.

\section{Missing Proofs}
\subsection{Proof of Proposition \ref{prop:NPhard-MNL}}

\begin{proof}
    Under the MNL choice model and when $T=1$, the objective function can be written as:
    \[
    F(\bx) = q_1 \frac{\sum_{i\in [m]} x_ie^{v_{1j}}}{U^c_1 + \sum_{i\in [m]} x_ie^{v_{1i}}} = q_1 -\frac{q_1U^c_1}{U^c_1 + \sum_{i\in [m]} x_ie^{v_{1i}}},
    \]
To maximize \( F(\bx) \), we need to maximize the sum \( U^c_1 + \sum_{i \in [m]} x_i e^{v_{1i}} \). Given the cardinality constraint \( \sum_{i \in [m]} x_i \leq r \), this can be achieved by sorting the list of utilities \( \{v_{11}, \ldots, v_{1m}\} \) in decreasing order and selecting the top \( r \) locations with the highest utilities. This sorting process can be performed in \( \mathcal{O}(m \log m) \) using a standard sorting algorithm.

For the case that $T=2$, let us prove the NP-hardness by connecting the problem with the set partition problem, which is known to be NP-hard. We first write the objective function as
\begin{align}
    F(\bx) &= q_1 \frac{\sum_{i\in [m]} x_ie^{v_{1j}}}{U^c_1 + \sum_{i\in [m]} x_ie^{v_{1i}}} + q_2 \frac{\sum_{i\in [m]} x_ie^{v_{2j}}}{U^c_1 + \sum_{i\in [m]} x_ie^{v_{2i}}} \label{eq:MNL-2type} \\
    &= (q_1 + q_2) - \left(\frac{U^c_1}{U^c_1 + \sum_{i\in [m]} x_i{V_{1i}}}  +\frac{U^c_2}{U^c_2 + \sum_{i\in [m]} x_i{V_{2i}}}\right)
\end{align}
where $V_{ti} = e^{v_{ti}}$, for all $t\in \{1,2\}, i\in [m]$, for notational simplicity. It  more convenient to define the objective function as a subset function as:
\[
F(S) = (q_1 + q_2) - \left(\frac{q_1U^c_1}{U^c_1 + \sum_{i\in S} {V_{1i}}}  +\frac{q_2U^c_2}{U^c_2 + \sum_{i\in S} {V_{2i}}}\right)
\]
 and convert the MCP as  a minimization problem:
 \begin{equation}
    \label{prob:FL-MNL} \tag{\sf $P^r$}
    \min~\left\{ G(S) = \frac{q_1U^c_1}{U^c_1 + \sum_{i\in S} {V_{1i}}}+ \frac{q_2U^c_2}{U^c_2 + \sum_{i\in S} {V_{2i}}}\Bigg|~ |S| = r\right\} 
    \end{equation}
    
To prove the hardness result, we will show that the \textit{Set Partition} problem can be converted into a set of $m$ instances of   \eqref{prob:FL-MNL} (so solving the Set Partition problem can be done by solving $m$ times \eqref{prob:FL-MNL}).  

Let us first introduce the Set Partition problem as below:

\paragraph{The Set Partition problem:}\textit{ There are $m$ integer numbers $a_1,\ldots,a_m$ such that $\sum_{j\in [m]}=2K$, is it possible to partition $[m]$ into two disjoint subsets $S_1, S_2$ such that $\sum_{j\in S_1}a_j = \sum_{j\in S_2}a_j = K$?}.

We now identify $m$  instances of \eqref{prob:FL-MNL} such that, by solving these $m$ instances, we solve the Set Partition problem above. For each $r\in [m]$, we chose a set of parameters  $\Theta^r$ as follows
\begin{equation}
    \Theta^r = \left\{ \Big(q_1,q_2,U^c_1,U^c_2, \{(V_{1j},V_{2j}),~ j\in  [m]\}\Big)~ \text{ s.t. }
    \begin{cases}
    V_{1j} = a_j; ~\forall j\in [m] \\
    V_{2j} = \max_j \{a_j\} - a_j +1;~\forall j\in [m] \\
    U^c_1 = r+ r \times \max_j \{a_j\} \\
    U^c_2 =  2K\\
     q_1= 1/U^c_1;~ q_2 =1/U^c_2\\
    \end{cases}
    \right\}
\end{equation}
So, for each set of parameters $\Theta^r$, the objective function becomes. 
   \begin{align*}       
     G(S) &= \frac{1}{r+ r\max_j \{a_j\} + \sum_{j\in S} a_j}+ \frac{1}{ 2K + \sum_{j\in S} (\max_j\{a_j\} +1 - a_j)}\\
     &=
     \frac{1}{r+ r\max_j \{a_j\} + \sum_{j\in S} a_j}+ \frac{1}{ 2K + r+ r\max_j\{a_j\} - \sum_{j\in S} a_j}
   \end{align*}
We now consider the following univariate function
$$
g(x) =\frac{1}{r+ r\max_j \{a_j\} + x} + \frac{1}{ 2K + r+ r\max_j\{a_j\} - x}  
$$
Taking its first-order derivative  we get
\begin{align*}
\frac{\partial g(x)}{\partial x} &= \frac{1}{ ( 2T + r+ r\max_j\{a_j\} - x)^2}   - \frac{1}{(r+ r\max_j \{a_j\} + x)^2} 
\end{align*}
Then we see that $\partial G(x)/\partial x = 0$ if and only if 
$$
2T + r+ r\max_j\{a_j\} - x = r+ r\max_j \{a_j\} + x
$$
or $x = K$. Moreover, $\partial g(x)/\partial x \leq 0$ if $x \geq K$, and  $\partial g(x)/\partial x \geq 0$ if $x \leq K$. Thus the function $g(x)$ achieves its unique maximum value at $x=K$.

We now get back to the problem $\max_{S}\{G(S)|~ |S|=r\}$. From the above results, we can see that if the partition problem has a solution, then solving $m$ problem \eqref{prob:FL-MNL} with $r = 1,\ldots,m$, each with the parameter set $\Theta^r$, will return at least one solution $\overline{S}$ such that $\sum_{j\in \overline{S}} = K$. On the other hand, if the Set Partition problem has no solution, then there is no solution $S$ to any problem instance in $\{\text{\eqref{prob:FL-MNL}}|~ r\in [m]\}$  such that $\sum_{j\in \overline{S}}a_j = K$. This confirms the reduction from the Set Partition problem into problem \eqref{prob:FL-MNL}, validating the {\sf NP-harness}.

\end{proof}

\subsection{Proof of Proposition \ref{prop:NL-poly}}
\begin{proof}
    Under the NL model, each facility/location can belong to only one nest. Thus, for each facility/location $i\in [m]$, there is only one nest $n\in [N]$ such that the membership values $\alpha_{in} = 1$, and  $\alpha_{in'}=0$ for all $n'\neq n$. The objective function can be written in a compact form as:
    \[
    F(\bx) = q_1\frac{\sum_{n\in [N]} W_n^{\sigma_n-1}(\sum_{i\in\cN_n} x_iV_i)}{\sum_{n\in [N]} W_n^{\sigma_n}}
    \]
   where \( V_{i} = V_{1in} \), \( W_{n} = W_{1n} \), \( \cN_n = \cN^1_n\) and \( \sigma_{n} = \sigma_{1n} \) for all \( i \in [m], n \in [N] \). Some subscripts are redundant and have been removed for simplicity of notation. This is possible because \( T=1 \) and for each option \( i \in [m] \), there is only one nest \( \cN_n \) such that \( i \in \cN_n \). Since 
   $W_n = \sum_{i \in \cN_n}x_iV_{i} + \sum_{i\in \cC \cap \cN_n} V_i$, we  can further write the objective function as:
   \[
   F(\bx) = q_1 U^c - q_1\frac{\sum_{n\in [N]} W_n^{\sigma_n-1}U^c_n}{\sum_{n\in [N]} W_n^{\sigma_n}}
  \]

where $U^c_n = \sum_{i\in \cC \cap \cN_n}  V_{i}$  (for notational simplicity). To prove that the problem \( \{\max_{\bx} F(\bx) \mid \bx \in \{0,1\}^m, \sum_{i \in [m]} x_i = r \} \) can be solved in polynomial time, we will construct a two-step approach. In the first step, we convert the problem into a bisection procedure, where each iteration involves solving a sub-problem with a simpler structure. In the second step, we demonstrate that each sub-problem can be solved in polynomial time.

For the first step, we write  convert the maximization problem into a minimization one as
\begin{equation}\label{eq-NL-min}
    \min_{\bx\in \{0,1\}^m} \left\{\frac{\sum_{n\in [N]} W_n^{\sigma_n-1}U^c_n}{\sum_{n\in [N]} W_n^{\sigma_n}}\Big|~ \sum_{i}x_i = r\right\}
\end{equation}
   which is equivalent to:
   \begin{equation}\label{eq:bisection-prob}
       \min \left\{\delta > 0|~\exists \bx \in \cX\text{ s.t. } \cG(\bx,\delta) = {\sum_{n\in [N]} W_n^{\sigma_n-1}U^c_n} - {\sum_{n\in [N]} \delta W_n^{\sigma_n}} \leq 0\right\}
   \end{equation}
We can solve \eqref{eq-NL-min} by performing a binary search over \(\delta\), where at each iteration we determine whether there is a binary solution \(\bx\) such that \(\sum_{i \in [m]} x_i = r\) and \(\cG(\bx, \delta) < 0\). If this is not the case, we decrease \(\delta\); otherwise, we increase it. At each iteration, to answer this question, we need to solve the following sub-problem:
\begin{equation}\label{eq:NL-subprob}
    \min_{\bx \in \cX} \{\cG(\bx, \delta)\}
\end{equation}
Using dynamic programming techniques, the following lemma demonstrates that we can solve \eqref{eq:NL-subprob} in polynomial time.

\begin{lemma}
The problem $\min_{\bx \in \cX} \{\cG(\bx, \delta)\}$ can be solve to optimality in $\cO(r^2m^2)$ for any $\delta>0$.
\end{lemma}
 \begin{proof}
   We write the objective function as 
\[ \cG(\bx, \delta) = \sum_{n \in [N]} \phi_n(W_n, \delta) \]
where \(\phi_n(W_n, \delta) = W_n^{\sigma_n - 1} - \delta W_n^{\sigma_n}\). Since \(\sigma_n \leq 1\) and \(\delta > 0\), we see that \(\phi_n(W_n)\) is decreasing in \(W_n\). Given the cardinality constraint \(\sum_i x_i = r\), we will formulate the sub-problem as the following dynamic program. Let \(N\) be the number of stages, where at each stage \(n \in [N]\) we need to select locations for nest \(n\). Let us define the following value function:
\[ \cV(n, w) = \min \left\{\sum_{l=n}^N \phi_n(W_n, \delta) \Big| \sum_{l=n}^N \sum_{i \in \cN_n \cup [m]} x_i \leq w \right\} \]
In other words, the value function \(\cV(n, w)\) represents the minimum value from nest \(n\) to nest \(N\) where there are \(w\) locations remaining to be selected. It can be seen that $$\cV(1, r) = \min_{\bx \in \cX} \left\{ \sum_{n \in [N]} \phi_n(W_n, \delta) \right\}.$$ Moreover, the value function \(\cV(n, w)\) satisfies the following Bellman equation:
 \begin{equation}
     \cV(n,w) = \begin{cases}
         \min_{u=0,\ldots, w}\left\{\min_{\substack{x_i,i\in \cN_n\cup [m]\\\sum_{i\in \cN_n\cup [m]}x_i \leq u}}\Phi_n(W_n,\delta) + \cV({n+1,w-u})\right\} & \forall n<N, w>0\\
         0 & \text{if  $n=N$ or $w=0$} 
     \end{cases}
 \end{equation}
 Moreover, the subproblem 
 \[
 \max_{\substack{x_i,i\in \cN_n\cup [m]\\\sum_{i\in \cN_n\cup [m]}x_i \leq u}}\Phi_n(W_n,\delta) 
 \]
This is equivalent to finding the maximum value of \(W_n\) such that \(W_n = \sum_{i \in \cN_n \cup [m]} x_i V_i + \sum_{i \in \cC \cap \cN_n} V_i\) and \(\sum_{i \in \cN_n \cup [m]} x_i \leq u\). This can be done easily by sorting \(V_i\) for all \(i \in \cN_n \cup [m]\) in decreasing order and selecting the top \(w\) highest values. This can be performed in \(\cO(|\cN_n \backslash \cC|^2)\). Thus, if the value function is given at stage \(n+1\), computing \(\cV(n, w)\) will take \(\cO(w(|\cN_n \backslash \cC|)^2)\). Consequently, the computation of \(\cV(n, w)\) for all \(w = 0, \ldots, r\) takes about \(\cO(r^2 |\cN_n \backslash \cC|^2)\).

To compute \(\cV(1, r)\), we can calculate \(\cV(n, w)\) backwards from \(\cV(N, w), \cV(N-1, w), \ldots\). At each stage, the computation can be done in \(\cO(r^2 |\cN_n \backslash \cC|^2)\), implying that solving the Bellman system will take about
\[ \cO\left(\sum_{n \in [N]} \cO(r^2 |\cN_n \backslash \cC|^2)\right) \]
Since \(|\cN_n \backslash \cC|^2\) can be bounded from above by \((\sum_{n} |\cN_n \backslash \cC|)^2 \leq m^2\), the Bellman equation can be solved in \(\cO(r^2 m^2)\). This completes the proof.

 \end{proof}

We now return to the main proof. The binary search guarantees that, given any \(\epsilon > 0\), after \(\log(1/\epsilon)\) iterations, it can find a \(\widehat{\delta}\) such that \(|\widehat{\delta} - \delta^*| \leq \epsilon\) and  
\[
\min_{\bx \in \cX} \cG(\bx, \widehat{\delta}) \leq 0 
\]
where \(\delta^*\) is the optimal solution to \eqref{eq:bisection-prob}, which is also the optimal value of the main problem \eqref{eq-NL-min}. This implies that there exists \(\widehat{\bx} \in \cX\) such that \(G(\widehat{\bx}, \widehat{\delta}) \leq 0\), or 
\[
 \delta^*  \leq \frac{\sum_{n \in [N]} W_n(\widehat{\bx})^{\sigma_n - 1} U^c_n}{\sum_{n \in [N]} W_n(\widehat{\bx})^{\sigma_n}} \leq \widehat{\delta} \leq \delta^* +  \epsilon
\]
where \(W_n(\widehat{\bx}) = \sum_{i \in [m]} \widehat{x}_i V_i + U^c_n\), which directly implies that
\[
\left| \frac{\sum_{n \in [N]} W_n(\widehat{\bx})^{\sigma_n - 1} U^c_n}{\sum_{n \in [N]} W_n(\widehat{\bx})^{\sigma_n}} - \delta^*\right| \leq \epsilon.
\]
As a result,  \(\widehat{\bx}\) is a \((\epsilon)\)-optimal solution of the maximization problem \(\max_{\bx \in \cX} F(\bx)\), i.e.,
\[
 F(\widehat{\bx}) \geq \max_{\bx \in \cX} F(\bx) - \epsilon
\]
as desired. Therefore, we can confirm that an \(\epsilon\)-optimal solution to \eqref{prob:CNL-MCP} can be found in \(\cO(m^2 r^2 \log(1/\epsilon))\).

\end{proof}

\section{Additional Experiments}

\subsection{Impact of CNL Parameters}

We conduct experiments to examine the impact of CNL parameters on the performance of our algorithms. In the CNL model, there are two key parameters influencing the complexity of the MCP-CNL: the dissimilarity parameter \(\pmb{\sigma}\) and the overlapping rate \(\gamma\). The dissimilarity parameter \(\pmb{\sigma}\) directly affects the nonlinearity of the objective function, distinguishing the CNL model from the simpler MNL model. On the other hand, the overlapping rate \(\gamma\) indicates the degree of overlap between nests, affecting how shared or distinct the locations are across different nests. By solving the MCP-CNL with varying \(\pmb{\sigma}\) and \(\gamma\), we aim to analyze how these parameters impact the complexity of the facility location problem. Specifically, we investigate how changes in \(\pmb{\sigma}\) affect the model's nonlinearity and how different levels of \(\gamma\) influence the correlation structure between locations. Through this analysis, we seek to understand the sensitivity of our algorithms to these parameters and identify any trends or patterns that could inform better decision-making in competitive facility location planning.

We select small-sized instances with \(T \leq 100\) from the \textbf{ORlib} and \textbf{HM14} datasets, setting the parameters \(\alpha\) and \(\beta\) to 1 and 0.01, respectively. For each instance, we vary the mean of the dissimilarity parameter \(\pmb{\sigma}\) and the overlapping rate \(\gamma\), and solve the corresponding MCP-CNL instances using \textbf{CP} with outer-approximation cuts. Since these instances are small, \textbf{CP} consistently terminates in less than 1 hour, returning optimal solutions. We then report the average computing time for each parameter setting.

\begin{figure}[!httb]
    \centering
    \begin{subfigure}[b]{0.45\textwidth}
        \centering
        \includegraphics[width=\linewidth]{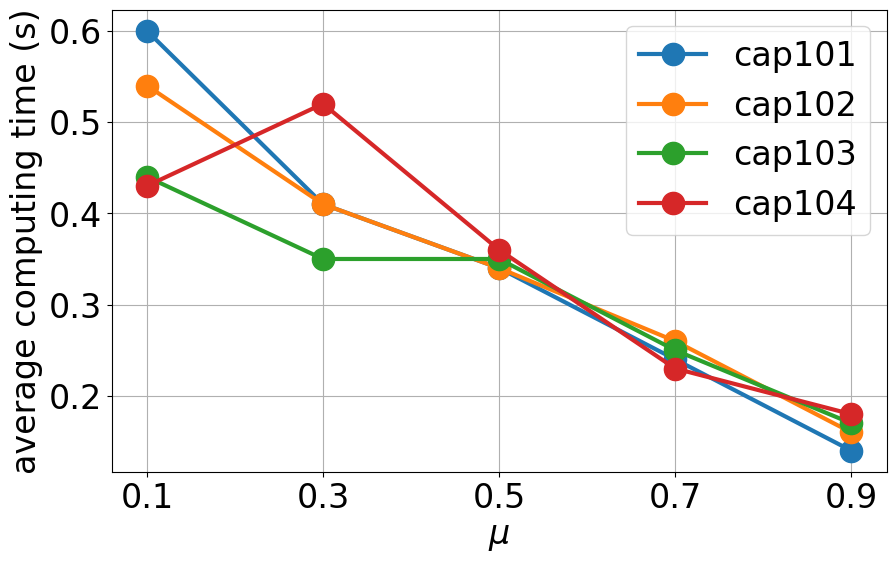}
        \caption{\textbf{ORlib} dataset with $(T,m) = (50,25)$.}
    \end{subfigure}
\hfill
    \begin{subfigure}[b]{0.45\textwidth}
        \centering
        \includegraphics[width=\linewidth]{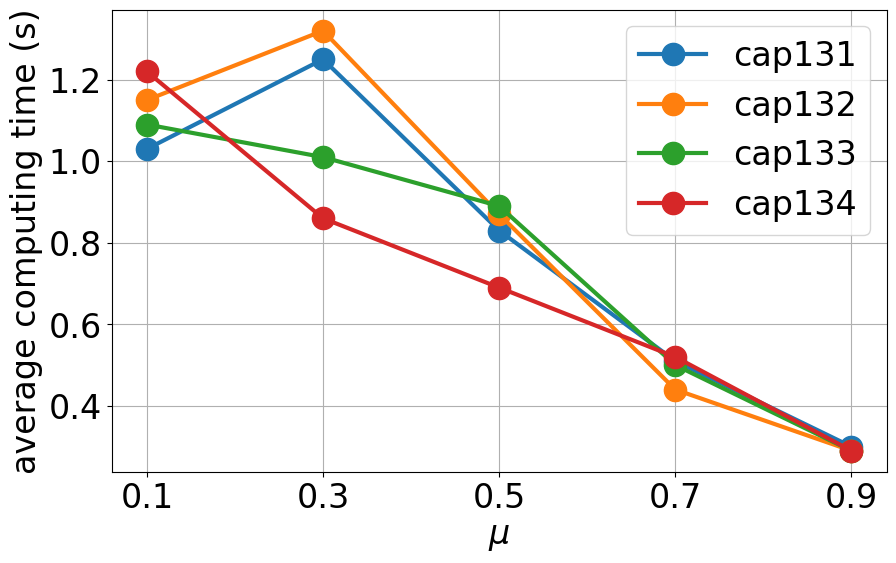}
        \caption{\textbf{ORlib} dataset with $(T,m) = (50,50)$.}
    \end{subfigure}
    \begin{subfigure}[b]{0.45\textwidth}
        \centering
        \includegraphics[width=\linewidth]{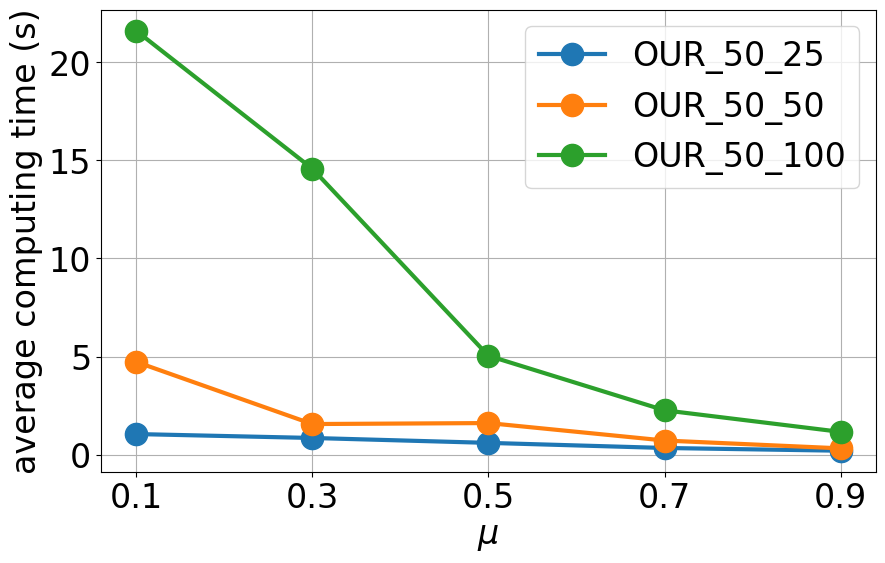}
        \caption{\textbf{HM14} dataset with $T=50$.}
    \end{subfigure}
\hfill
    \begin{subfigure}[b]{0.45\textwidth}
        \centering
        \includegraphics[width=\linewidth]{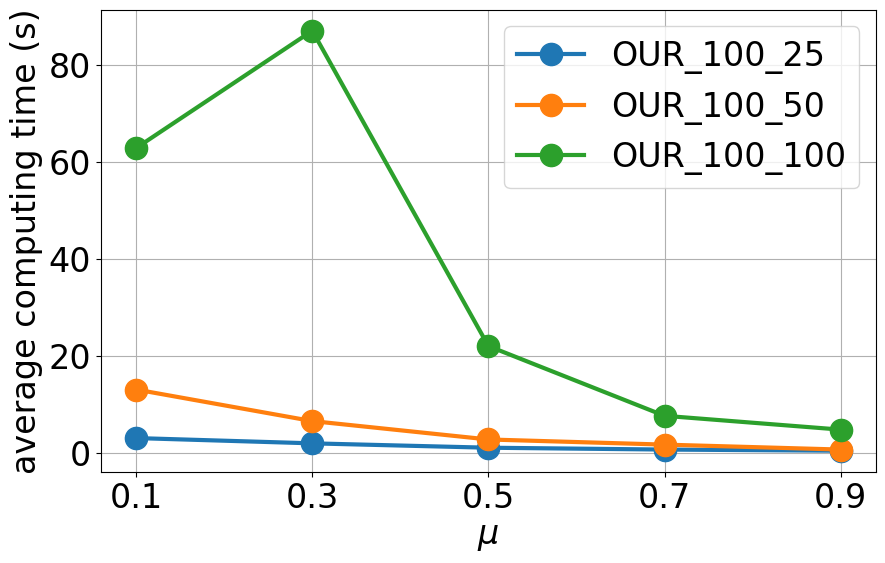}
        \caption{\textbf{HM14} dataset with $T=100$.}
    \end{subfigure}
    \caption{Average computing time of CP algorithm with different distribution of $\pmb{\sigma}$.}
    \label{fig:sigma}
\end{figure}

\begin{figure}[!httb]
    \centering
    \begin{subfigure}[b]{0.45\textwidth}
        \centering
        \includegraphics[width=\linewidth]{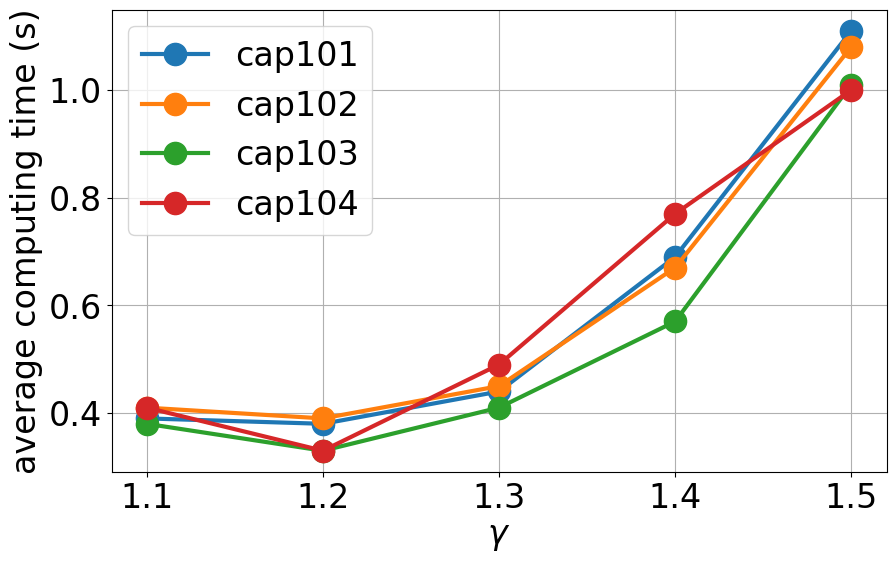}
        \caption{\textbf{ORlib} dataset with $(T,m) = (50,25)$.}
    \end{subfigure}
\hfill
    \begin{subfigure}[b]{0.45\textwidth}
        \centering
        \includegraphics[width=\linewidth]{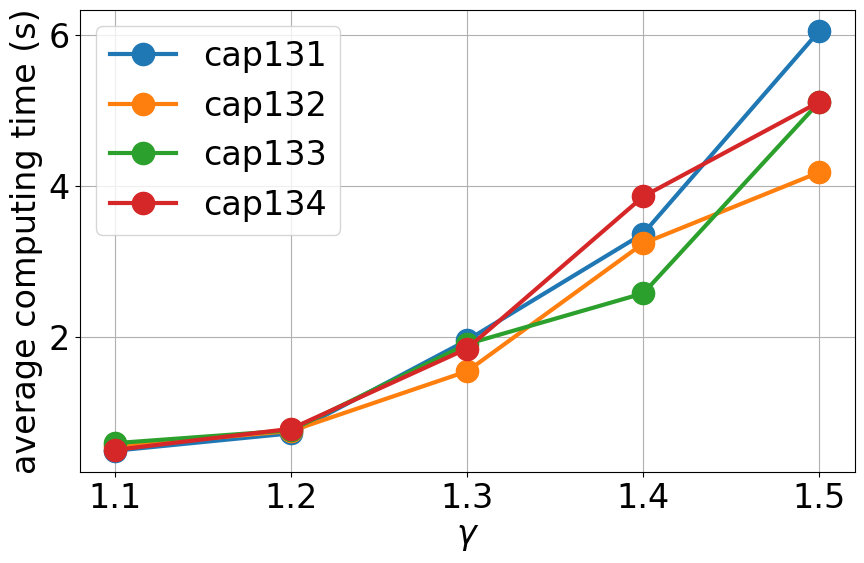}
        \caption{\textbf{ORlib} dataset with $(T,m) = (50,50)$.}
    \end{subfigure}
    \begin{subfigure}[b]{0.45\textwidth}
        \centering
        \includegraphics[width=\linewidth]{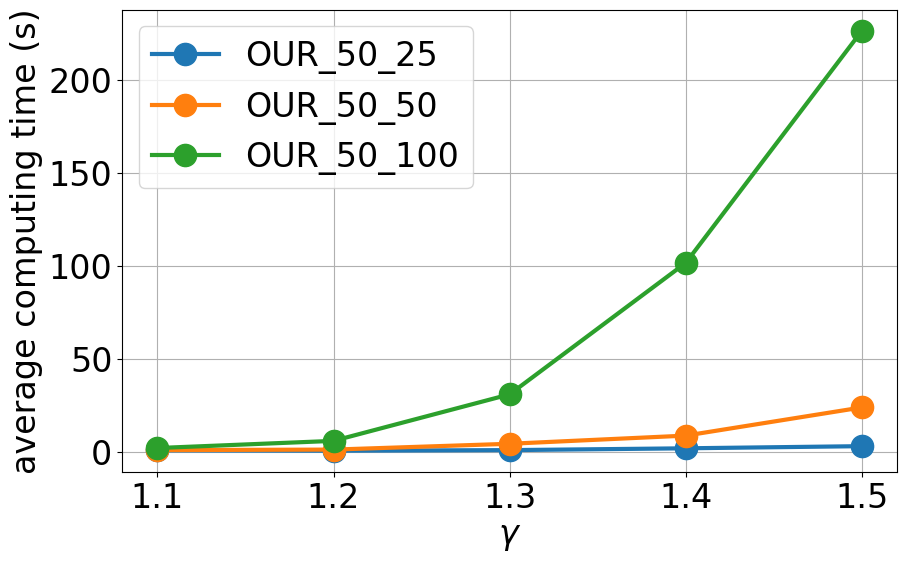}
        \caption{\textbf{HM14} dataset with $T=50$.}
    \end{subfigure}
\hfill
    \begin{subfigure}[b]{0.45\textwidth}
        \centering
        \includegraphics[width=\linewidth]{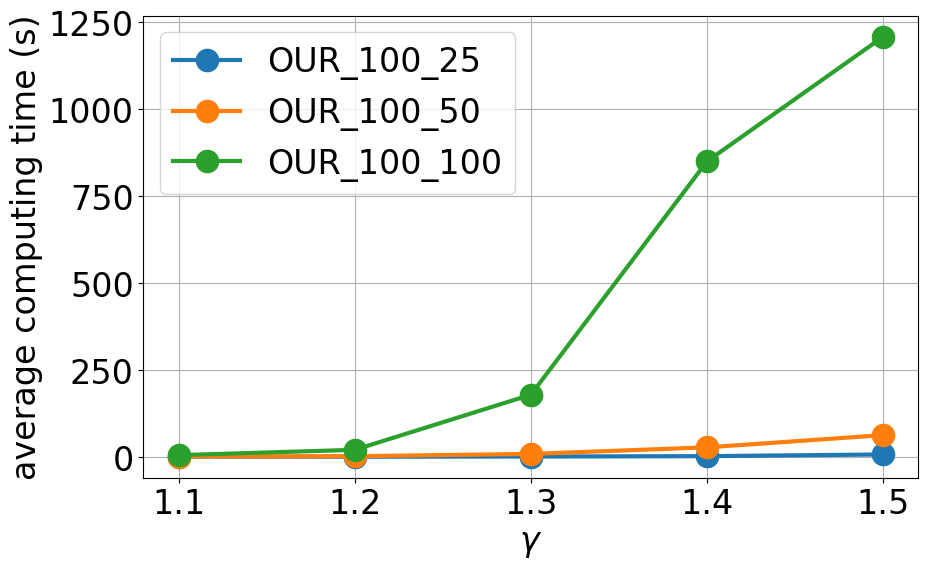}
        \caption{\textbf{HM14} dataset with $T=100$.}
    \end{subfigure}
    \caption{Average computing time of CP algorithm with different overlapping rates $\gamma$.}
    \label{fig:gamma}
\end{figure}

Since the dissimilarity parameters \(\pmb{\sigma}\) are drawn from a normal distribution with a mean \(\mu\), we vary the mean parameter within the set \(\{0.1, 0.3, 0.5, 0.7, 0.9\}\). For each pair of \((\mu, \gamma)\), we generate 10 instances for the experiments. Consequently, each parameter set now corresponds to 450 different instances. Figure \ref{fig:sigma} reports the average computing time of the \textbf{CP} algorithm on the selected instances when \(\mu\) varies from 0.1 to 0.9. The figure shows that the average computation time is almost inversely proportional to the dissimilarity parameters, i.e., higher \(\pmb{\sigma}\) values result in lower computing times. This result aligns with expectations, as when \(\pmb{\sigma}\) approaches 1, the CNL model becomes closer to the MNL model, making the objective function of the MCP-CNL more closer to a linear fractional program, which is simpler to solve.

To analyze the impact of the overlapping parameter, we vary \(\gamma\) within the set \(\{1.1, 1.2, 1.3, 1.4, 1.5\}\). For each pair of \((\gamma, r)\), we generate 10 instances for the experiments, resulting in a total of 450 instances for each parameter set. Figure \ref{fig:gamma} plots the average computing time as \(\gamma\) increases, showing that the computing time tends to increase as \(\gamma\) gets higher. In other words, instances require more time to solve with higher overlapping rates. This observation aligns with expectations, as higher \(\gamma\) implies that each location can be shared among more nests, resulting in a more complex objective function. It is worth noting that when \(\gamma = 1\), the CNL model collapses to an NL model where each location cannot belong to more than one nest. In this case, as shown in Proposition \ref{prop:NL-poly}, the MCP can be solved in polynomial time when there is only one customer type. Figure \ref{fig:gamma} also indicates that, in this situation, the algorithm can solve the instances very quickly.

Figures \ref{fig:sigma} and \ref{fig:gamma} together demonstrate that the MCP becomes more complex and difficult to solve as the overlapping rate \(\gamma\) increases or the dissimilarity parameter \(\pmb{\sigma}\) decreases. This generally aligns with the fact that the CNL model approaches its simpler instances (MNL or NL) when \(\gamma\) decreases or \(\pmb{\sigma}\) increases.

\subsection{Comparison of Outer-Approximation and Submodular Cuts}
\begin{table}[!ht]
    \centering
    \caption{Comparison Results for outer-approximation and submodular cuts in \textbf{B\&C}.}
    \label{tab:results-bc}
    \begin{tabular}{cccccccccc}
    \hline
        \multirow{2}{*}{Problem} & \multirow{2}{*}{$T$} & \multirow{2}{*}{$m$} & \multicolumn{2}{c}{\#Best} &  \multicolumn{2}{c}{\#Opt} &  \multicolumn{2}{c}{\#Time} \\ \cmidrule(lr){4-5}\cmidrule(lr){6-7}\cmidrule(lr){8-9}
        ~ & ~ & ~ & \multicolumn{1}{c}{\textbf{OA}} & \multicolumn{1}{c}{\textbf{OA+SC}} & \multicolumn{1}{c}{\textbf{OA}} & \multicolumn{1}{c}{\textbf{OA+SC}} & \multicolumn{1}{c}{\textbf{OA}} & \multicolumn{1}{c}{\textbf{OA+SC}} \\ \hline
        cap101 & 50 & 25 & \textbf{81} & \textbf{81} & \textbf{81} & \textbf{81} & \textbf{0.11} & \textbf{0.11} \\ 
        cap102 & 50 & 25 & \textbf{81} & \textbf{81} & \textbf{81} & \textbf{81} & 0.10 & \textbf{0.09} \\ 
        cap103 & 50 & 25 & \textbf{81} & \textbf{81} & \textbf{81} & \textbf{81} & \textbf{0.10} & \textbf{0.10} \\ 
        cap104 & 50 & 25 & \textbf{81} & \textbf{81} & \textbf{81} & \textbf{81} & \textbf{0.09} & \textbf{0.09} \\ 
        cap131 & 50 & 50 & \textbf{81} & \textbf{81} & \textbf{81} & \textbf{81} & \textbf{0.23} & 0.26 \\ 
        cap132 & 50 & 50 & \textbf{81} & \textbf{81} & \textbf{81} & \textbf{81} & \textbf{0.12} & 0.14 \\ 
        cap133 & 50 & 50 & \textbf{81} & \textbf{81} & \textbf{81} & \textbf{81} & \textbf{0.12} & \textbf{0.12} \\ 
        cap134 & 50 & 50 & \textbf{81} & \textbf{81} & \textbf{81} & \textbf{81} & \textbf{0.13} & 0.14 \\ 
        capa & 1000 & 100 & 77 & \textbf{80} & \textbf{76} & \textbf{76} & 1582.02 & \textbf{1389.03} \\ 
        capb & 1000 & 100 & \textbf{81} & \textbf{81} & \textbf{79} & \textbf{79} & 1296.73 & \textbf{1262.18} \\ 
        capc & 1000 & 100 & 68 & \textbf{71} & 48 & \textbf{52} & \textbf{2121.76} & 2155.58 \\ 
        OUR & 50 & 25 & \textbf{81} & \textbf{81} & \textbf{81} & \textbf{81} & 0.18 & \textbf{0.16} \\ 
        OUR & 50 & 50 & \textbf{81} & \textbf{81} & \textbf{81} & \textbf{81} & \textbf{0.18} & 0.19 \\ 
        OUR & 50 & 100 & \textbf{81} & \textbf{81} & \textbf{81} & \textbf{81} & \textbf{0.33} & 0.34 \\ 
        OUR & 100 & 25 & \textbf{81} & \textbf{81} & \textbf{81} & \textbf{81} & \textbf{0.31} & 0.34 \\ 
        OUR & 100 & 50 & \textbf{81} & \textbf{81} & \textbf{81} & \textbf{81} & \textbf{0.47} & 0.48 \\ 
        OUR & 100 & 100 & \textbf{81} & \textbf{81} & \textbf{81} & \textbf{81} & \textbf{2.58} & 2.60 \\ 
        OUR & 200 & 25 & \textbf{81} & \textbf{81} & \textbf{81} & \textbf{81} & \textbf{1.96} & 1.99 \\ 
        OUR & 200 & 50 & \textbf{81} & \textbf{81} & \textbf{81} & \textbf{81} & 2.47 & \textbf{2.01} \\ 
        OUR & 200 & 100 & \textbf{81} & \textbf{81} & \textbf{81} & \textbf{81} & \textbf{15.12} & 15.47 \\ 
        OUR & 400 & 25 & \textbf{81} & \textbf{81} & \textbf{81} & \textbf{81} & 12.42 & \textbf{7.72} \\ 
        OUR & 400 & 50 & \textbf{81} & \textbf{81} & \textbf{81} & \textbf{81} & 15.13 & \textbf{6.89} \\ 
        OUR & 400 & 100 & \textbf{81} & \textbf{81} & \textbf{81} & \textbf{81} & 134.97 & \textbf{113.13} \\ 
        OUR & 800 & 25 & \textbf{81} & \textbf{81} & \textbf{81} & \textbf{81} & 76.83 & \textbf{75.86} \\ 
        OUR & 800 & 50 & \textbf{81} & \textbf{81} & \textbf{81} & \textbf{81} & 159.61 & \textbf{135.40} \\ 
        OUR & 800 & 100 & 80 & \textbf{81} & \textbf{80} & 79 & 1205.19 & \textbf{1128.49} \\ 
        NYC & 2000 & 59 & 65 & \textbf{76} & 64 & \textbf{72} & 1944.93 & \textbf{511.92} \\ \hline\hline
        Average & ~ & ~ & 79.74	& 80.41	& 78.85 & 79.26 & 317.56 & \textbf{252.25}\\\hline
    \end{tabular}
\end{table}
In this section, we present experiments to assess the impact of outer-approximation and submodular cuts on the performance of the \textbf{CP} and \textbf{B\&C} approaches. We compare the performance of the \textbf{B\&C} and \textbf{CP} algorithms when using only outer-approximation cuts (denoted as \textbf{OA}) and when combining outer-approximation cuts with submodular cuts (denoted as \textbf{OA+SC}). Results for \textbf{CP} or \textbf{B\&C} with only submodular cuts are not included, as this combination generally performs poorly compared to those with only outer-approximation cuts or with both types of cuts \citep{Ljubic2018outer}.

We implemented these algorithms on the same instances used in our main comparison in Section \ref{sec:experiments}. The numerical results for the \textbf{B\&C} and \textbf{CP} algorithms are reported in Table \ref{tab:results-bc} and Table \ref{tab:results-cp}, respectively. For the \textbf{B\&C} algorithm, both the average number of best solutions and the number of instances solved to optimality increase when submodular cuts are added. Moreover, the submodular cuts improve the average computing time (252.25 seconds compared to 317.56 seconds without submodular cuts), especially for \textbf{NYC} dataset, as shown in Table \ref{tab:results-bc}.

For the \textbf{CP} algorithm, as presented in Table \ref{tab:results-cp}, submodular cuts help tighten the bounds during the convergence process, improving the discovery of best solutions, while solving to optimality is only slightly less efficient ($76.41\%$ compared to $77\%$ without submodular cuts). In terms of average running time, submodular cuts improve computing time for the \textbf{ORlib} dataset; however, adding additional cuts can lead to higher computing times, as observed with the \textbf{HM14} dataset.

\begin{table}[!ht]
    \centering
    \caption{Comparison results for outer-approximation and submodular cuts in \textbf{CP}.}
    \label{tab:results-cp}
    \begin{tabular}{ccccccccc}
    \hline
        \multirow{2}{*}{Problem} & \multirow{2}{*}{$T$} & \multirow{2}{*}{$m$} & \multicolumn{2}{c}{\#Best} &  \multicolumn{2}{c}{\#Opt} &  \multicolumn{2}{c}{\#Time} \\\cmidrule(lr){4-5}\cmidrule(lr){6-7}\cmidrule(lr){8-9}
        ~ & ~ & ~ & \multicolumn{1}{c}{\textbf{OA}} & \multicolumn{1}{c}{\textbf{OA+SC}} & \multicolumn{1}{c}{\textbf{OA}} & \multicolumn{1}{c}{\textbf{OA+SC}} & \multicolumn{1}{c}{\textbf{OA}} & \multicolumn{1}{c}{\textbf{OA+SC}} \\ \hline
        cap101 & 50 & 25 & \textbf{81} & \textbf{81} & \textbf{81} & \textbf{81} & \textbf{0.28} & 0.36 \\ 
        cap102 & 50 & 25 & \textbf{81} & \textbf{81} & \textbf{81} & \textbf{81} & \textbf{0.28} & 0.32 \\ 
        cap103 & 50 & 25 & \textbf{81} & \textbf{81} & \textbf{81} & \textbf{81} & \textbf{0.53} & 0.59 \\ 
        cap104 & 50 & 25 & \textbf{81} & \textbf{81} & \textbf{81} & \textbf{81} & \textbf{0.32} & 0.41 \\ 
        cap131 & 50 & 50 & \textbf{81} & \textbf{81} & \textbf{81} & \textbf{81} & \textbf{1.19} & 1.29 \\ 
        cap132 & 50 & 50 & \textbf{81} & \textbf{81} & \textbf{81} & \textbf{81} & \textbf{0.62} & 0.69 \\ 
        cap133 & 50 & 50 & \textbf{81} & \textbf{81} & \textbf{81} & \textbf{81} & \textbf{0.67} & 0.85 \\ 
        cap134 & 50 & 50 & \textbf{81} & \textbf{81} & \textbf{81} & \textbf{81} & \textbf{0.63} & 0.66 \\ 
        capa & 1000 & 100 & \textbf{81} & \textbf{81} & 74 & \textbf{76} & 1705.72 & \textbf{1615.53} \\ 
        capb & 1000 & 100 & 74 & \textbf{75} & \textbf{64} & 60 & 1885.85 & \textbf{1850.24} \\ 
        capc & 1000 & 100 & 56 & \textbf{62} & \textbf{31} & 26 & 3100.84 & \textbf{3099.91} \\ 
        OUR & 50 & 25 & \textbf{81} & \textbf{81} & \textbf{81} & \textbf{81} & \textbf{0.82} & 0.92 \\ 
        OUR & 50 & 50 & \textbf{81} & \textbf{81} & \textbf{81} & \textbf{81} & \textbf{1.16} & 1.38 \\ 
        OUR & 50 & 100 & \textbf{81} & \textbf{81} & \textbf{81} & \textbf{81} & \textbf{2.08} & 2.54 \\ 
        OUR & 100 & 25 & \textbf{81} & \textbf{81} & \textbf{81} & \textbf{81} & \textbf{1.28} & 2.49 \\ 
        OUR & 100 & 50 & \textbf{81} & \textbf{81} & \textbf{81} & \textbf{81} & \textbf{3.71} & 4.45 \\ 
        OUR & 100 & 100 & \textbf{81} & \textbf{81} & \textbf{81} & \textbf{81} & \textbf{18.73} & 21.41 \\ 
        OUR & 200 & 25 & \textbf{81} & \textbf{81} & \textbf{81} & \textbf{81} & 21.53 & \textbf{19.08} \\ 
        OUR & 200 & 50 & \textbf{81} & \textbf{81} & \textbf{81} & \textbf{81} & 6.81 & \textbf{6.42} \\ 
        OUR & 200 & 100 & \textbf{81} & \textbf{81} & \textbf{81} & \textbf{81} & 245.97 & \textbf{226.64} \\ 
        OUR & 400 & 25 & \textbf{81} & \textbf{81} & \textbf{81} & \textbf{81} & \textbf{18.45} & 18.76 \\ 
        OUR & 400 & 50 & \textbf{81} & \textbf{81} & \textbf{81} & \textbf{81} & \textbf{34.23} & 43.06 \\ 
        OUR & 400 & 100 & 79 & \textbf{81} & 78 & \textbf{79} & \textbf{536.33} & 538.85 \\ 
        OUR & 800 & 25 & \textbf{81} & \textbf{81} & \textbf{81} & \textbf{81} & \textbf{67.50} & 143.43 \\ 
        OUR & 800 & 50 & \textbf{81} & \textbf{81} & \textbf{81} & \textbf{81} & \textbf{168.64} & 189.47 \\ 
        OUR & 800 & 100 & \textbf{76} & \textbf{76} & \textbf{62} & 61 & \textbf{1852.07} & 1977.20 \\ 
        NYC & 2000 & 59 & \textbf{77} & 76 & \textbf{69} & 60 & \textbf{1394.81} & 1678.24 \\ \hline\hline
        Average & ~ & ~ & 79.41 & \textbf{79.70} & \textbf{77.00} & 76.41 & \textbf{410.04} & 423.90\\\hline
    \end{tabular}
\end{table}

\end{document}

\end{document}